\documentclass[reqno]{amsart}
\usepackage{amsmath,amsfonts,amsthm,amsopn,color,amssymb,enumitem}
\usepackage{palatino}
\usepackage{graphicx}
\usepackage[english]{babel}
\usepackage{caption}
\usepackage{subcaption}
\usepackage[colorlinks=true]{hyperref}
\hypersetup{urlcolor=blue, citecolor=red, linkcolor=blue}
\usepackage[utf8]{inputenc}

\usepackage{esint}
\usepackage[title]{appendix}

\newtheorem{theorem}{Theorem}[section]
\newtheorem{lemma}[theorem]{Lemma}

\newtheorem{proposition}[theorem]{Proposition}
\newtheorem{remark}[theorem]{Remark}

\def\beq{\begin{equation}}
\def\eeq{\end{equation}}
\def\ba{\begin{array}}
\def\ea{\end{array}}
\def\S{\mathbb S}
\def\R{\mathbb R}

\def\Sk{\mathbb{S}^{d}(k)}
\def\e{\varepsilon}

\def\dis{\displaystyle}

\newcommand{\rmnote}[1]{}

\numberwithin{equation}{section}
\newenvironment{abs}{\textbf{Abstract}\mbox{  }}{ }
\newenvironment{key words}{\textbf{Keywords}\mbox{  }}{ }

\begin{document}

\title{Overdetermined elliptic problems in nontrivial contractible domains of the sphere}

\author{David Ruiz}
\address{(D.~Ruiz)
	IMAG, Departamento de An\'alisis matem\'atico, Universidad de Granada,
	Campus Fuente-nueva, 18071 Granada, Spain} \email{daruiz@ugr.es}

\author{Pieralberto Sicbaldi}
\address{(P.~Sicbaldi)
	Departamento de An\'alisis matem\'atico,
	Universidad de Granada,
	Campus Fuentenueva,
	18071 Granada,
	Spain \& Aix Marseille Universit\'e - CNRS, Centrale Marseille - I2M, Marseille, France}
\email{pieralberto@ugr.es}

\author{Jing Wu}
\address{(J.~Wu)
	Departamento de An\'alisis matem\'atico, Universidad de Granada,
	Campus Fuente-nueva, 18071 Granada, Spain} \email{jingwulx@correo.ugr.es}

\thanks{D. R. has been supported by the FEDER-MINECO Grant PGC2018-096422-B-I00 and by J. Andalucia (FQM-116).
 P. S has been supported by the FEDER-MINECO Grant PID2020-117868GB-I00 and by J. Andalucia Grant P18-FR-4049. J. W. has been supported
by the China Scholarship Council (CSC201906290013) and by J. Andalucia (FQM-116).
D. R. and P. S. also acknowledge financial support from the Spanish Ministry of Science and Innovation (MICINN), through the \emph{IMAG-Maria de Maeztu} Excellence Grant CEX2020-001105-M/AEI/10.13039/501100011033.}
\maketitle

\noindent
\begin{abs}
In this paper, we prove the existence of nontrivial contractible domains $\Omega\subset\mathbb{S}^{d}$, $d\geq2$, such that the overdetermined elliptic problem
\begin{equation*}
  \begin{cases}
  -\varepsilon\Delta_{g} u +u-u^{p}=0 &\mbox{in $\Omega$, }\\
   u>0 &\mbox{in $\Omega$, }\\
  u=0 &\mbox{on $\partial\Omega$, }\\
  \partial_{\nu} u=\mbox{constant} &\mbox{on $\partial\Omega$, }
  \end{cases}
\end{equation*}
admits a positive solution. Here $\Delta_{g}$ is the Laplace-Beltrami operator in the unit sphere $\mathbb{S}^{d}$ with respect to the canonical round metric $g$, $\varepsilon>0$ is a small real parameter and $1<p<\frac{d+2}{d-2}$ ($p>1$ if $d=2$). These domains are perturbations of $\mathbb{S}^{d}\setminus D,$ where $D$ is a small geodesic ball. This shows in particular that Serrin's theorem for overdetermined problems in the Euclidean space cannot be generalized to the sphere even for contractible domains.
\end{abs}\\
\begin{key words}: Overdetermined boundary conditions; semilinear elliptic problems; bifurcation theory.
\end{key words}
 \indent



\section{Introduction}
\label{Section 1}
Semilinear overdetermined elliptic problems in the form
\begin{equation}\label{eq03}
\left\{\begin{array} {ll}
\Delta u + f(u) = 0 & \mbox{in }\; \Omega,\\
 u> 0 & \mbox{in }\; \Omega,\\
               u= 0 & \mbox{on }\; \partial \Omega, \\
 \partial_{\nu} u=\mbox{constant} &\mbox{on }\; \partial
\Omega,\,
\end{array}\right.
\end{equation}
have received much attention in the last decades. Typically $\Omega$ is a regular domain in $\R^d$, $f$ is a Lipschitz function and $\partial_{\nu} u$ is the derivative of $u$ in the direction of the outward normal unit vector $\nu$ on the boundary $\partial \Omega$. These problems are called ``overdetermined" because of the two boundary conditions, and appear quite naturally in many different phenomena in Physics, see \cite{S02,S56} for more details.

\medskip

If $\Omega \subset \R^d $ is bounded the problem has been completely solved by J. Serrin, who proved in 1971 that if \eqref{eq03} is solvable, then $\Omega$ must be a ball (see the original work \cite{S71} and also \cite{PS07} for more details). Serrin's proof is based on the {\it moving plane method}, introduced in 1956 by A. D. Alexandrov in \cite{A56} to prove that the only compact, connected, embedded hypersurfaces in $\R^d$ with constant mean curvature are the spheres. This proof showed an analogy between overdetermined elliptic problems and constant mean curvature surfaces. From that moment, the moving plane method has become a very important tool in Analysis to obtain symmetry results for solutions of semilinear elliptic equations.

\medskip

Starting from Serrin's result, two natural lines of research have been considered: first, the case of unbounded domains $\Omega$ in the Euclidean space, and second, the case of domains $\Omega \subset M$, where $M$ is a Riemannian manifold.

\medskip

The first topic is motivated from the fact that overdetermined elliptic problems arise naturally in free boundary problems, when the variational structure imposes suitable conditions on the separation interface (see, for example, \cite{AC81}). As is well known, several methods for studying the regularity of the interface are based on blowup techniques that lead to the study of an overdetermined elliptic problem in an unbounded domain. In this framework, H. Berestycki, L. Caffarelli and L. Nirenberg ~\cite{BCN97} stated the following conjecture:
\medskip

\textbf{BCN Conjecture (1997)}. Assume that $\Omega$ is a smooth domain with connected complement, then the existence of a bounded solution to problem (\ref{eq03}) for some Lipschitz function $f$ implies that $\Omega$ is either a ball, a half-space, a generalized cylinder $B^{k}\times\mathbb{R}^{d-k}$ ($B^{k}$ is a ball in $\mathbb{R}^{k}$), or the complement of one of them.

\medskip

Many works have shown that under some assumptions on the function $f$ or on the domain $\Omega$ the conjecture is true, see \cite{FV10, HLSWW, HHP11, R97, RRS17, RS13, T13}.
Nevertheless, the conjecture is false in its generality and was disproved for $d\geq3$ in \cite{S10}, where the second author found a periodic perturbation of the straight cylinder $B^{d-1}\times\mathbb{R}$ that supports a periodic solution to the problem (\ref{eq03}) with $f(u)=\lambda u, \lambda>0$. After such first construction, other examples of nontrivial solutions have been obtained, see for instance \cite{DPW15, FMW17, LWW, RSW21, SS12}. In all these examples the boundary of the domain has a shape that looks like an unbounded constant mean curvature surface, showing again an important analogy with those surfaces. Another type of counterexample to the conjecture, of particular interest for this paper, has been given in \cite{RRS20}. In that work it is shown that \eqref{eq03} admits a solution for some nonradial exterior domains (i.e. the complement of a compact region in $\mathbb{R}^d$ that is not a closed ball), for a suitable function $f(u)$. In dimension 2, this represents this first construction of a counterexample to the BCN conjecture, that turns to be false in any dimension. It is worth pointing out that such result breaks the analogy with the theory of constant mean curvature surfaces.
\medskip


Overdetermined elliptic problems (\ref{eq03}) posed in complete Riemannian manifolds instead of the Euclidean setting have also been a natural field of research. In this framework, we need to replace in \eqref{eq03} the classical Laplacian by the Laplace-Beltrami operator $\Delta_g$ associated to the metric $g$ of the manifold $M$:
\begin{gather}\label{eq03bis}
\begin{cases}
  \Delta_g u +f(u)=0 &\mbox{in $\Omega$, } \\
  u >0 &\mbox{in $\Omega$, } \\
  u=0 &\mbox{on $\partial \Omega$, }\\
  \partial_{\nu} u=\mbox{constant} &\mbox{on $\partial \Omega$, }
  \end{cases}
\end{gather}
where $\Omega$ is a domain of $M$, and $\nu$ is the normal outward unit vector about $\partial \Omega$ with respect to $g$. For general Riemannian manifolds, solutions of overdetermined elliptic problems of the form \eqref{eq03bis} are obtained in \cite{DS15, DEP19, FMV13, MS16, PS09, S14}.

\medskip

It is clear that any symmetry result on the solutions of \eqref{eq03bis} tightly depends on the symmetry of the ambient manifold. In a given arbitrary manifold, geodesic balls are not domains where \eqref{eq03bis} can be solved. As shown in \cite{DS15}, for small volumes it is possible to construct solutions of \eqref{eq03bis} with $f(u) = \lambda\, u$ in perturbations of geodesic balls centered at specific points of the manifold, but such domains in general are not geodesic balls. In fact, one can expect to obtain a Serrin-type result only for manifolds that are symmetric in a suitable sense. More precisely (see also the introduction of \cite{DEP19}), the key ingredient for the moving plane method is the use of the reflexion principle in any point and any direction. For this we need that for any $p\in M$ and any two vectors $v,w \in T_pM$ there exists an isometry of $M$ leaving $p$ fixed and transporting $v$ into $w$ (i.e. $M$ is isotropic) and that such isometry is induced by the reflection with respect to a hypersurface. Such last hypersurface must be totally geodesic, being the set of fixed points of an isometry. But an isotropic manifold admitting totally geodesic hypersurfaces must have constant sectional curvature (see for example \cite{BCO03}, p. 295) and the only isotropic Riemannian manifolds of constant sectional curvature are the Euclidean space $\mathbb{R}^d$, the round sphere $\mathbb{S}^d$, the hyperbolic space $\mathbb{H}^d$ and the real projective space $\mathbb{RP}^d$ (see \cite{S91}). Now, domains in $\mathbb{RP}^d$ arise naturally to domains in its universal covering $\mathbb{S}^d$, so we are left to consider our problem in $\mathbb{R}^d$, $\mathbb{S}^d$ and $\mathbb{H}^d$.

\medskip

Being the problem in $\mathbb{R}^d$ completely understood by Serrin, the framework of the other two space form manifolds has been treated in 1998 in the paper by S. Kumaresan and J. Prajapat \cite{KP98}. In the case of $\mathbb{H}^d$ they obtained a complete counterpart of the Serrin's theorem: namely, by using the moving plane method, they show that if  $\Omega$ is a bounded domain of $\mathbb{H}^d$ and \eqref{eq03bis} admits a solution, then $\Omega$ must be a geodesic ball. The case of $\mathbb{S}^d$ is different. In fact, even if the reflexion principle is valid in any point, one needs to have a totally geodesic hypersurface that does not intersect the domain in order to start the moving plane. This is not a problem in $\mathbb{R}^d$ or $\mathbb{H}^d$, but it is in $\mathbb{S}^d$. Since the totally geodesic hypersurfaces are the equators, one can start the moving plane method if and only if the domain is contained on a hemisphere. And this is exactly the case considered in \cite{KP98}: if $\Omega$ is contained in a hemisphere and  \eqref{eq03bis} admits a solution, then $\Omega$ must be a geodesic ball.

\medskip

Other natural domains of $\mathbb{S}^d$ where \eqref{eq03bis} has solutions are symmetric neighborhoods of any equator. Such symmetric annuli are not contractible and their existence comes from the geometry of $\mathbb{S}^d$ in the same way as they exist in a cylinder or in a torus. Moreover, perturbations of such domains in $\mathbb{S}^d$ where \eqref{eq03bis} still admits a solution have been built in \cite{FMW18} in the same way as this has been done for the same kind of domains in cylinders or in tori \cite{S10}.

\medskip

Taking these facts in account, the following question arises naturally: is it true that if $\Omega \subset \mathbb{S}^d$ is contractible and \eqref{eq03bis} can be solved, then $\Omega$ must be a geodesic ball? In \cite{EM19}, J.M. Espinar and L. Mazet give an affirmative answer to this question if $d=2$ but under some extra assumptions on the nonlinear term $f(u)$. The proof of such result shows again an analogy between overdetermined elliptic problems and constant mean curvature surfaces, because it is highly inspired by the proof of the Hopf's Theorem that states that the only immersed constant mean curvature surfaces of genus zero in $\mathbb{R}^3$ are the spheres.
\medskip

In this paper we show that the answer to the previous question is negative: there exist contractible domains $\Omega \subset \mathbb{S}^d$, different from geodesic balls, where  \eqref{eq03bis} can be solved for some nonlinearities $f$. This construction works for any dimension $d \geq 2$. In view of \cite{KP98}, such domains cannot be contained in any hemisphere. Our main result can be stated as follows.

 \begin{theorem} \label{Th01}
Let $d\in\mathbb{N},d\geq2$ and $1<p<\frac{d+2}{d-2}$\,($p>1$ if $d=2).$ Then there exist domains $D$, which are perturbations of a small geodesic ball, such that the problem
 \begin{gather}\label{eq02}
  \begin{cases}
  -\varepsilon\Delta_{g} u +u-u^{p} =0&\emph{in $\mathbb{S}^{d}\setminus D$ }, \\
   u >0&\emph{in $\mathbb{S}^{d}\setminus D$ }, \\
  u=0 &\emph{on $\partial D$ },\\
  \partial_{\nu} u=\emph{constant} &\emph{on $\partial D$ },
  \end{cases}
\end{gather}
admits a solution for some $\varepsilon>0$.
 \end{theorem}
\begin{figure}[htbp]
\centering
\includegraphics[width=5.5cm]{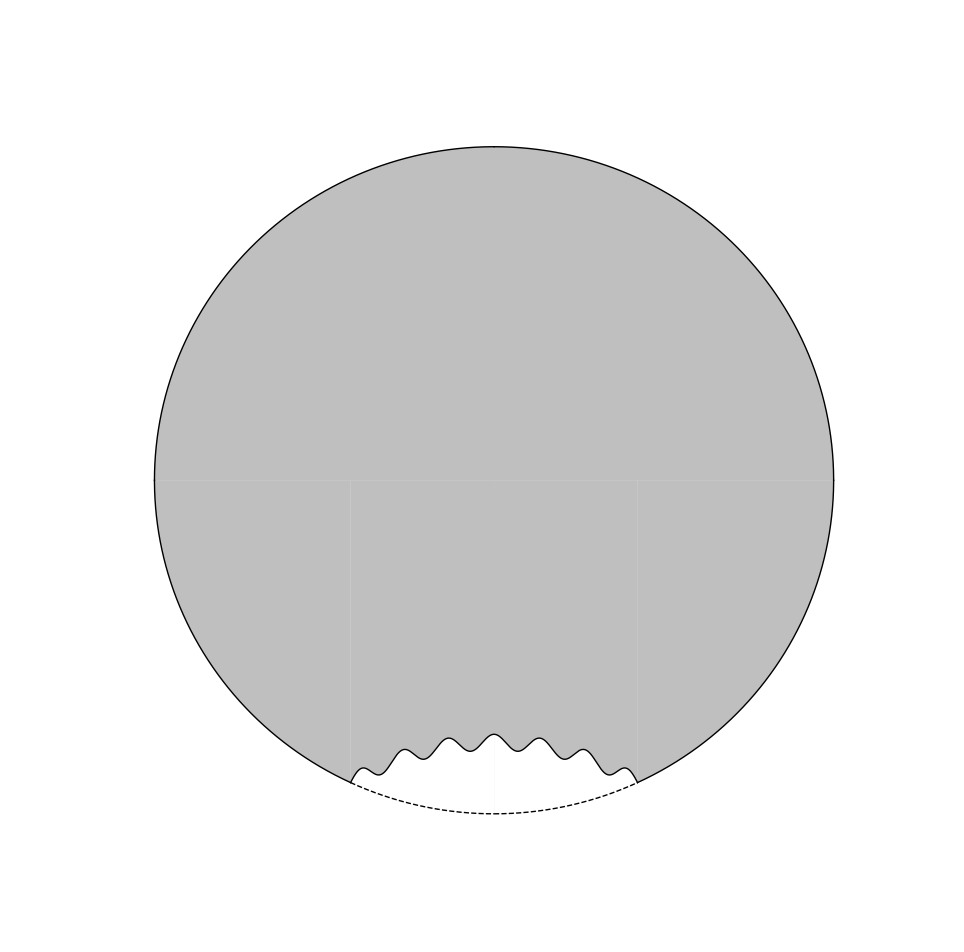}
\caption{The domain $\mathbb{S}^{d}\setminus D$}
\end{figure}

 A more precise formulation of the above result will be given in Section 2.

 \medskip

 The main idea of the proof is the following. First, one uses a dilation to pass to a problem posed in $\mathbb{S}^{d}(k)$, the sphere of radius $1/k$, where $k$ will be a small parameter. We take a geodesic ball of radius 1 in $\mathbb{S}^{d}(k)$, and consider now $\Omega_k$ the complement of such ball in $\mathbb{S}^{d}(k)$. The main idea is that, as $k\to 0$, the domain $\Omega_k$ converges (in a certain sense) to the exterior domain $\mathbb{R}^d \setminus B(0,1)$. Thanks to the result in \cite{RRS20}, we have the existence of nontrivial solutions of \eqref{eq03} for a suitable choice of $f$, bifurcating from a family of regular solutions $u_\lambda$ of the problem:
  \begin{gather}
 	\begin{cases}
 		-\lambda \Delta u +u-u^{p} =0&\emph{in $\R^{d} \setminus B(0,1)$ }, \\
 		 u >0&\emph{in $\mathbb{R}^{d}\setminus B(0,1)$ }, \\
 		u=0 &\emph{on $\partial B(0,1)$ },\\
 		\partial_{\nu} u=\emph{constant} &\emph{on $\partial B(0,1)$} .
 	\end{cases}
 \end{gather}
In this paper we first show that for $k$ sufficiently small there exists a similar family of solutions posed in $\Omega_k$. This is accomplished by making use of a (quantitative) Implicit Function Theorem. Then we study the behavior of the linearized operator by using a  perturbation argument, and taking into account the case studied in \cite{RRS20}. In such way, we can use the Krasnoselskii bifurcation theorem to show the existence of a branch of nontrivial solutions to $\eqref{eq03bis}$.

In our arguments, we rely on the study of the linearized operator given in \cite{RRS20}. At a certain point the assumption $p < \frac{d+2}{d-2}$ is needed in \cite{RRS20}, and hence our result is also restricted to that case. Moreover, such assumption is needed also in order to get $L^\infty$ uniform estimates on the solutions.

 \medskip

The rest of the paper is organized as follows. In Section \ref{Section 2} we set the notations and we give a more precise statement of Theorem \ref{Th01}. The existence of a radial family of solutions in $\Omega_k$ is shown in Section \ref{Section 3}. In Section \ref{Section 4} we construct the nonlinear Dirichlet-to-Neumann operator and compute its linearization under certain nondegeneracy assumptions. These hypotheses will be verified in Section \ref{Section 5}, and that section is also devoted to studying the properties of the linearized operator computed in Section \ref{Section 4}. With all those ingredients, we can use a local bifurcation argument to prove our main result; this is done in Section \ref{Section 6}.

\bigskip

{\bf Acknowledgment:} The authors wish to express their sincere gratitude to the referee for his/her careful reading and the many comments, which have definitively improved the quality of this work.

\section{Notations and statement of the main result}
\label{Section 2}

If $k>0$, let $\mathbb{S}^{d}(k)$ be the $d$-dimensional sphere of radius $\frac{1}{k}$ naturally embedded in $\mathbb{R}^{d+1}$ ($d\ge2$). We consider $\mathbb{S}^{d}(k)$ as a Riemannian manifold with the metric $g_k$ endowed by its embedding in $\mathbb{R}^{d+1}$. The sectional curvature of such manifold is equal to $k^{2}$.
When $k=1$ we write directly $\mathbb{S}^{d}$ as usual. We fix two opposite points $S,N \in \mathbb{S}^{d}(k)$ (let's say respectively the south and the north pole) and we use the exponential map of $\mathbb{S}^{d}(k)$ centered at $S$,
$$\mbox{exp}_S: B\left(0, \frac{\pi}{k}\right) \to \Sk \setminus\{N\},$$
where $B\left(0, \frac{\pi}{k}\right) \subset \R^d$ is the Euclidean ball of radius $\frac{\pi}{k}$ centered at the origin.
Given any continuous function $v: \mathbb{S}^{d-1} \to \left(0, \frac{\pi}{k}\right)$, we define the domain
\[  B_v = \mbox{exp}_S \left (  \left\{x\in\mathbb{R}^{d}:0\leq|x|<v\left(\frac{x}{|x|}\right)\right\} \right )\subset \Sk\,.
\]
The precise statement of our result is the following:

\begin{theorem} \label{Th11}
	Let $d\in\mathbb{N}$, $d\geq2$, let $1<p<\frac{d+2}{d-2}$ ($p>1$ if $d=2$). Then, there exists a real number $k_{0}>0$, such that for any $0<k<k_{0}$ the following holds true: there exist a sequence of real parameters $\lambda_{m} = \lambda_m(k)$ converging to some $\lambda_*(k) >0$, a sequence of nonconstant functions $v_{m} = v_m (k)\in C^{2,\alpha}(\mathbb{S}^{d-1})$ converging to $0$ in $C^{2,\alpha}$, and a sequence of positive functions $u_m\in C^{2,\alpha}(\mathbb{S}^{d}(k)\setminus B_{1+v_{m}})$, such that the equation
	\begin{gather}\label{eq102}
		\begin{cases}
			-\lambda_{m}\Delta u_m +u_m-u_m^{p} =0&\emph{in $\mathbb{S}^{d}(k)\setminus B_{1+v_{m}}$ }, \\
			u_m=0 &\emph{on $\partial B_{1+v_{m}}$ },\\
			\partial_{\nu} u_m=\emph{constant} &\emph{on $\partial B_{1+v_{m}}$ },
		\end{cases}
	\end{gather}
is satisfied.
\end{theorem}

Observe that Theorem \ref{Th01} follows at once from the previous result, by a scale change transforming $\mathbb{S}^{d}(k)$ into $\mathbb{S}^{d}$.

\medskip

Throughout the paper we shall use the coordinates in $\Sk$ given by the exponential map centered at the south pole composed with polar coordinates in $\R^d$. In other words, we write:
\begin{equation} \label{coord}  \begin{array}{c} X: \left[0, \frac{\pi}{k}\right) \times \S^{d-1} \to \Sk \setminus \{N\} \\
X(r, \theta) = \mbox{exp}_S(r\,  \theta). \end{array} \end{equation}
Observe that $X$ is well defined but singular at the south pole $S$.

\medskip
With this notation, the set $B_v$ can be written as:
$$ B_v=\left\{(r, \theta) \in \left[0, \frac{\pi}{k}\right) \times \S^{d-1}: \  r < v(\theta) \right\}.$$ Moreover, the standard metric $g_k$ on $\Sk$ and its corresponding Laplace-Beltrami operator can be written in $(r, \theta)$ coordinates as:
\[g_k=dr^{2}+S_{k}^{2}(r)d\theta^{2},\]
\[\Delta:=\Delta_{g_{k}}=\partial^{2}r+(d-1)\frac{C_{k}(r)}{S_{k}(r)}\partial_r+\frac{1}{S_{k}^{2}(r)}\Delta_{\mathbb{S}^{d-1}},\]
where \begin{equation} \label{cos-sin} S_{k}(r)=\frac{\sin(kr)}{k}, \ C_{k}(r)=\cos(kr).\end{equation}
See for instance \cite{C84}.

Sometimes it will be useful to consider $S_k(r)$, $C_k(r)$ defined for all $r>0$; in such case, $S_k(r)= C_k(r)=0$ for all $r \geq \pi/k$.

\medskip
In order to prove Theorem \ref{Th11} we will make use of symmetry groups. Given a group of isometries acting on $\mathbb{S}^{d-1}$, we say that $\Omega\subset\mathbb{S}^{d}(k)$ is $G$-symmetric if, working in coordinates:
\[
(r, \theta) \in \Omega \Rightarrow (r, g(\theta)) \in \Omega
\]
for any $g \in G$. Let us point out that in the case in which $G = \mathcal{G}$ the group of all possible symmetries of $\S^{d-1}$, then $G$-symmetry is just axial symmetry in $\Sk$ (with respect to the axis of $\mathbb{R}^{d+1}$ passing through the south and the north poles).

\medskip

We define:
\begin{align*}
	C^{k,\alpha}_{G}(\mathbb{S}^{d-1})&=\{u\in C^{k,\alpha}(\mathbb{S}^{d-1}):u=u\circ g ~\forall g\in G\}.
\end{align*}
For later purposes we also define the set of functions in $C^{k,\alpha}_{G}(\S^{d-1})$ whose mean is $0$:
\[ C^{k,\alpha}_{G,0}(\mathbb{S}^{d-1})= \left \{u\in C^{k,\alpha}_G(\mathbb{S}^{d-1}):\ \int_{\S^{d-1}} u=0 \right \}.	 \]

\medskip

If $\Omega\subset\mathbb{S}^{d}(k)$ is $G$-symmetric, we define the following H\"{o}lder spaces of $G$-symmetric functions:
\begin{align*}
C^{k,\alpha}_{G}(\Omega)&=\{u\in C^{k,\alpha}(\Omega):u(r, \theta)=u(r, g(\theta)) \ \forall g\in G\}.
\end{align*}
In addition, we denote the Sobolev spaces of $G$-symmetric functions as follows:
\begin{align*}
  H^{1}_{G}(\Omega)&=\{u\in H^{1}(\Omega): \ u(r, \theta)=u(r, g(\theta)) \ \forall g\in G\}, \\
   H^{1}_{0,G}(\Omega)&=\{u\in H^{1}_{0}(\Omega):\ u(r, \theta)=u(r, g(\theta)) \ \forall g\in G\}.
\end{align*}

\medskip

Let us recall that the norm of the Sobolev space $H^1(\Omega)$ is given by:
$$ \| u \|_{H^1(\Omega)} = \left (\int_{\Omega}\left(|\nabla u|^{2}+u^{2}\right) \textnormal{dvol}_{g_{k}}\right )^{1/2}$$
where the gradient depends on the metric $g_k$.
If $\Omega = \Sk \setminus B_v$, we can write this expression in coordinates $(r, \theta)$ as:
$$ \int_{\Omega}\left(|\nabla u|^{2}+u^{2}\right) \textnormal{dvol}_{g_k}=\int_{v(\theta)}^{\pi/k} \int_{\mathbb{S}^{d-1}} S^{d-1}_{k}(r)\left[(\partial_{r} u)^{2}+\frac{|\nabla_{\theta}u|^{2}}{S^2_{k}(r)}+u^{2}\right] d\theta \, d r.$$
In the axisymmetric case of $ \Sk$ with respect to the two poles (that is, $G= \mathcal{G}$ the group of all possible isometries in $\S^{d-1}$) we write a subscript $r$ instead of $G$, to highlight that the spaces depend only on the $r$ variable. Hence, we shall write:
$$H_{r}^1(\Sk \setminus B_1) = \left \{u: \left[1, \frac{\pi}{k}\right)\to \R, \ \| u \|_k < \infty \right \},$$
where
\begin{equation} \label{radialnorm} \| u \|_k = \left ( \int_{1}^{\pi/k} S^{d-1}_{k}(r) \left[(\partial_{r} u)^{2}+ u^{2}\right] dr \right )^{1/2}. \end{equation}
Clearly,
$$ \| u \|_{H^1_{r}(\Sk \setminus B_1)}=  \sqrt{\omega_{d-1}}\, \|u \|_k,$$
where $\omega_{d-1}$ is the $(d-1)$-dimensional measure of the unit sphere $\mathbb{S}^{d-1}$.

\medskip

Moreover we shall make use of the notation:
$$H_{0,r}^1(\Sk \setminus B_1) = \left \{u \in H_{r}^1(\Sk \setminus B_1), \ u(1)=0 \right \}.$$
Observe that in the above definitions the functions $u(r)$ are absolutely continuous in $\left[1, \frac{\pi}{k}\right)$ (possibly singular at $\frac{\pi}{k}$).

\medskip

We denote by $\mu_{i}=i(i+d-2),i\in\mathbb{N}$ the eigenvalues of the Laplace-Beltrami operator $\Delta_{\S^{d-1}}$ in $\mathbb{S}^{d-1}$. From now on, we shall fix a symmetry group $G$ in $\S^{d-1}$ satisfying the following property:
\begin{itemize}
  \item [(G)] Defining by $\{\mu_{i_{l}}\}_{l\in\mathbb{N}}$ the eigenvalues of $\Delta_{\mathbb{S}^{d-1}}$ restricted to $G$-symmetric functions and by $m_{l}$ their multiplicities, we require $i_{1}\geq2$ and $m_{1}$ odd.
\end{itemize}

A group satisfying those properties is the dihedric group $\mathbb{D}_n$, $n \geq 2$, if $d=2$. For $d>2$ one can take for instance $G=O(2) \times O(d-2)$. Other examples are possible, see \cite[Remark 2.2]{RRS20}.

\section{Existence of the axisymmetric solution to the Dirichlet problem}
\label{Section 3}
As mentioned before, we will use a local bifurcation argument. On that purpose, we first need to build an axially symmetric solution to the problem:
\begin{equation}\label{eq13}
  \begin{cases}
  -\lambda\Delta u+u-u^{p}=0 &\mbox{in $\mathbb{S}^{d}(k)\setminus B_{1}$, }\\
  u=0 &\mbox{on $\partial B_{1}$, }
   \end{cases}
\end{equation}
for suitable values of $\lambda$.
This is the goal of this section. We will do it for $k$ sufficiently small (i.e. for spheres $\Sk$ with large radius, and then small curvature).

\medskip

By using the coordinates $(r, \theta)$ as in \eqref{coord}, we have to find a solution $u(r)$ of the ODE problem:
\begin{equation}\label{eq14}
  \begin{cases}
  -\lambda\left[\partial_{r}^{2}+(d-1)\frac{C_{k}(r)}{S_{k}(r)}\partial_{r}\right]u
  +u-u^{p}=0 & r \in \left(1,\frac{\pi}{k}\right), \\
  u(1)=0, & \\
  u'({\frac{\pi}{k}})=0, &
  \end{cases}
\end{equation}
where $C_k(r)$ and $S_k(r)$ are defined in \eqref{cos-sin}. Observe that, for any fixed $r>0$,
$$ \frac{C_{k}(r)}{S_{k}(r)} \to \frac{1}{r} \mbox{ as } k \to 0.$$
Hence, at least formally, a limit problem for \eqref{eq14} as $k \to 0$ is:
\begin{equation} \label{preradial}
	\begin{cases}
		-\lambda\left(\partial_{r}^{2}+\frac{d-1}{r}\partial_{r}\right)u
		+u-u^{p} =0&r >1, \\
		u(1)=0. &
	\end{cases}
\end{equation}
Those are just radially symmetric functions of the Dirichlet problem:
\begin{equation} \label{radial}
	\begin{cases}
		-\lambda \Delta u+u-{u}^{p}=0 &\mbox{in $\mathbb{R}^{d}\setminus B_{1}$, }\\
		{u} =0 &\mbox{on $\partial B_{1}$ }.
	\end{cases}
\end{equation}
In the proposition below we list some known properties of this problem.

\begin{proposition} \label{list} We have:
	\begin{enumerate}
		\item[a)] For any $\lambda>0$, there exists a positive radially symmetric $C^{\infty}$ solution of \eqref{radial}. This solution increases in the radius up to a certain maximum, and then it decreases and converges to $0$ at infinity exponentially.
		\item[b)] Such positive and radial solution to \eqref{radial} is unique: we denote it by $\tilde{u}_{\lambda}$. Moreover it has an exponential decay (see \cite{strauss}, for instance):
		
		$$ \tilde{u}_{\lambda}(x) \sim |x|^{\frac{d-1}{2}} e^{- \frac{1}{\sqrt{\lambda}} |x|} \mbox{ as } |x| \to +\infty.$$
		\item[c)] Set $B_1^c = \mathbb{R}^{d}\setminus B_{1}$, let $H^1_{0,r}(B_1^c)$ be the classical Sobolev space $H^1_0(B_1^c)$ restricted to radial functions, and let $H^{-1}_{r}(B_1^c)$ be its dual. Let us define the linearized operator $L_\lambda: H^1_{0,r}(B_1^c) \to H^{-1}_{r}(B_1^c)$,
		\begin{equation}\label{linear-d}
			L_\lambda(\phi) = - \lambda \Delta \phi + \phi  - p \tilde{u}_{\lambda}^{p-1} \phi \,,
		\end{equation}
		and consider the eigenvalue problem:
		$$ L_\lambda(\phi) = \tau\, \phi.$$
		This problem has a unique negative eigenvalue and no zero eigenvalues. In other words, $\tilde{u}_{\lambda}$ is nondegenerate in $H^1_{0,r}(B_1^c)$ and has Morse index 1. We denote by $\tilde{z}_{\lambda} \in H^1_{0,r}(B_1^c)$ (normalized by $\| \tilde{z}_{\lambda} \|=1$) the positive eigenfunction with negative eigenvalue, i.e.
		\begin{equation} \label{z}
			\left\{\begin{array} {ll}
				-\lambda\, \Delta \tilde{z}_{\lambda} + \tilde{z}_{\lambda}  - p \tilde{u}_{\lambda}^{p-1} \tilde{z}_{\lambda} = \tilde{\tau}_{\lambda} \tilde{z}_{\lambda} & \mbox{in }\; B_1^c, \\
				\tilde{z}_{\lambda} = 0 & \mbox{on }\; \partial B_1, \\
			\end{array}\right.
		\end{equation}
		where $\tilde{\tau}_{\lambda}<0$. Moreover $\tilde{z}_{\lambda}$ is a $C^{\infty}$ function.
	\end{enumerate}
\end{proposition}
Statement a) is quite well known and has been proved in \cite{esteban}, for instance. The results b) and c) are more recent and have been obtained in \cite{f-m-tanaka, tang}.

\medskip

We now state the main result of this section.
\begin{proposition} \label{Pr21} Given $\e\in (0,1)$, there exists $k_0>0$ such that for any $k \in (0, k_0)$ and any $\lambda \in [\e, 1/\e]$ there exists a positive solution $u_{k, \lambda}$ to the problem (\ref{eq14}). Moreover, for any $\lambda \in [\e, 1/\e]$,
	
\begin{equation} \label{limitu} \lim_{k \to 0} \| u_{k, \lambda} - \tilde{u}_{\lambda}\|_k =0, \end{equation}
where $\tilde{u}_{\lambda}$ is the unique positive radial solution of \eqref{radial} and $\| \cdot \|_{k}$ is given by \eqref{radialnorm}.
\end{proposition}

In order to prove the Proposition \ref{Pr21}, we use the following version of the Inverse Function Theorem. We include its proof in the Appendix for the sake of completeness.
\begin{proposition} \label{Pr22} Let $Y$ be a Hilbert space, $v\in Y$ and $F\in C^{1}(Y,Y)$. Suppose that:
\begin{itemize}
		\item[(A1)] $\|F(v)\|<\delta$ for some fixed $\delta>0;$
		\item[(A2)] The derivative operator $F':Y\rightarrow Y$ is invertible and $\|F'(v)^{-1}\|\leq c,$ for some $c>0;$
		\item[(A3)] Define $U=\{z\in Y:\|z\|\leq 2 c \, \delta\}$ and assume that:
\[\|F'(v+z)-F'(v)\|<\frac{1}{2 c},\forall z\in U.\]
\end{itemize}
Then there exists a unique $z\in U$ such that $F(v+z)=0.$ Moreover, $\|F'(v+z)^{-1}\| \leq 2c$.
\end{proposition}

In order to apply the previous proposition to our purpose we will need the following lemma, whose proof again is postponed to the Appendix.
\begin{lemma} \label{Le202} For any $u\in H_{0}^{1}(\mathbb{S}^{d}(k)\setminus B_{1}),0<k< \frac{\pi^{2}}{2}$, we have
\[\|u\|_{L^{s}}\leq C\|u\|_{H^{1}},\]
where $2<s\leq2^{\ast}=\frac{2d}{d-2}$ if $d\geq3,$ $s>2$ if $d=2$, and $C=C(d,s)>0$ is a constant independent of $k$.
\end{lemma}

With these two results we are now able to prove Proposition \ref{Pr21}.

\begin{proof}[Proof of Proposition \ref{Pr21}] We will apply the Proposition \ref{Pr22} in the Sobolev space
\begin{equation}\label{Y}
  Y= H_{0,r}^{1}(\mathbb{S}^{d}(k)\setminus B_{1})
\end{equation}
via the coordinates given in \eqref{coord}. We define the operator
\[
F(u)=\Phi(\tilde{F}(u)):Y\rightarrow Y,
\]
where $\Phi:Y^{*}\rightarrow Y$ is the isomorphism given by the Riesz Representation Theorem and the functional
$\tilde{F}:Y\rightarrow Y^{*}$ is defined by
\[\tilde{F}(u)w=\omega_{d-1}\int_{1}^{\frac{\pi}{k}}S^{d-1}_{k}(r)(\lambda u'w'+uw-(u^{+})^{p}w) \, dr.\]

\medskip

Observe that if $F(u)=0$, then $u$ is a solution of problem \eqref{eq13} by the maximum principle. In what follows we fix $\lambda \in [\e, 1/\e]$ for some $\varepsilon \in (0,1)$ and define $v:=v_{k,\lambda}=\tilde{u}_\lambda\chi_{k}\geq 0 ,$ where $\tilde{u}_\lambda$ is the positive solution of \eqref{preradial} and $0 \leq \chi_k \leq 1$ is a cut-off function such that:
$$ \left \{ \begin{array}{ll} \chi_{k}(r)=1 &  r \in \left(1,\frac{\pi}{\sqrt{k}}\right), \\ \chi_{k}(r)=0 &  r \in \left(\frac{2\pi}{\sqrt{k}},\frac{\pi}{k}\right), \\ |\chi_{k}'(r)| \leq  k^{-1/2} &  r \in \left(1,\frac{\pi}{k}\right). \end{array} \right. $$
This function will play the role of $v$ in Proposition \ref{Pr22}. Now, let us verify that the assumptions of Proposition \ref{Pr22} are satisfied. For that, the exponential decay of $\tilde{u}_\lambda$ will be essential. First, it is clear that $\tilde{F}$ is a $C^1$ map.

%

\medskip

\textbf{(A1)} Let's prove that $\|\tilde{F}(v)\| = o_{k}(1).$ We compute
 \begin{align*}
 | \frac{\tilde{F}(v)w}{\omega_{d-1}}| &=\left|\int_{1}^{\frac{\pi}{k}}S^{d-1}_{k}(r)(\lambda v'w'+vw-v^{p}w)dr\right|\\
     &=\left|\int_{1}^{\frac{\pi}{k}}S^{d-1}_{k}(r)\left[\lambda (\tilde{u}_\lambda\chi_{k})'w'+(\tilde{u}_\lambda\chi_{k}w-\tilde{u}_\lambda^{p}\chi_{k}^{p}w+\tilde{u}_\lambda^{p}\chi_{k}w-\tilde{u}_\lambda^{p}\chi_{k}w)\right]dr\right|\\
    &= \Bigg| \int_{1}^{\frac{\pi}{k}}S^{d-1}_{k}(r)\Bigg[\lambda (\tilde{u}_\lambda\chi_{k}'w'- \tilde{u}_\lambda'\chi_{k}'w)+\tilde{u}_\lambda^{p}(\chi_{k} -\chi_{k}^{p})w\,  \\
 &\quad \qquad +(d-1)\lambda \tilde{u}_\lambda'\chi_{k}w\left(\frac{1}{r}-\frac{C_{k}(r)}{S_{k}(r)}\right)\Bigg]\, dr \,\Bigg|\\
    &\leq \left|\int_{1}^{\frac{\pi}{\sqrt{k}}}(d-1)\lambda \tilde{u}_\lambda'wS^{d-1}_{k}(r)\left(\frac{1}{r}-\frac{C_{k}(r)}{S_{k}(r)}\right)\, dr\, \right| \\
    &\quad\qquad +\Bigg|\int_{\frac{\pi}{\sqrt{k}}}^{\frac{\pi}{k}}S^{d-1}_{k}(r)\Bigg[\lambda (\tilde{u}_\lambda\chi_{k}'w'- \tilde{u}_\lambda'\chi_{k}'w)+\tilde{u}_\lambda^{p}(\chi_{k}
     -\chi_{k}^{p})w\\ &\quad \qquad \qquad  +(d-1)\lambda \tilde{u}_\lambda'\chi_{k}w\left(\frac{1}{r}-\frac{C_{k}(r)}{S_{k}(r)}\right)\Bigg]\, dr \, \Bigg|\\
    &\leq c_{1}\max\{A,B\}\, \|w\|_{k}+c_{2}\frac{1}{\varepsilon}\left(\frac{ e^{-\frac{\pi}{\sqrt{k\lambda}}}}{k^{\frac{d-1}{4}}}+
    \frac{e^{-\frac{p\pi}{\sqrt{k\lambda}}}}{k^{\frac{d-1}{2}}}
    +\frac{ e^{-\frac{\pi}{\sqrt{k\lambda}}}}{k^{\frac{d-3}{2}}}\right)\|w\|_{k}\\
     &\leq c_{1}\max\{A,B\}\, \|w\|_{k}+c_{2}\frac{1}{\varepsilon}\left(\frac{ e^{-\frac{\sqrt{\varepsilon}\pi}{\sqrt{k}}}}{k^{\frac{d-1}{4}}}+
    \frac{e^{-\frac{p\sqrt{\varepsilon}\pi}{\sqrt{k}}}}{k^{\frac{d-1}{2}}}
    +\frac{ e^{-\frac{\sqrt{\varepsilon}\pi}{\sqrt{k}}}}{k^{\frac{d-3}{2}}}\right)\|w\|_{k},,
\end{align*}
where $A=\left|\frac{k}{\tan(k)}-1\right|,B=\left|1-\frac{\sqrt{k}\pi}{\tan(\sqrt{k}\pi)}\right|.$
\medskip

\textbf{(A2)} Let's prove that $\tilde{F}'$ is invertible and $\|\tilde{F}'(v)^{-1}\|\leq c,c>0.$  By contradiction, we suppose that either $\tilde{F}'(v)$ is not invertible or $\|\tilde{F}'(v)^{-1}\|$ is not bounded. We first assume that $\tilde{F}'(v)$ is not invertible, then we can have that $\tilde{F}'(v)$ is not injective (see \cite[Lemma 4.1]{AEKS14}).
Suppose that $\tilde{F}'(v)$ is injective, let us write $\tilde{F}'(v)(\psi_{1},\psi_{2})=(\mathcal{A}\psi_{1},\psi_{2})$ for all $\psi_{1},\psi_{2}\in Y$ that is given in (\ref{Y}), where the operator
\begin{align*}
  \mathcal{A}:Y &\rightarrow Y^{*} \\
  \psi&\mapsto-\lambda\Delta\psi+\psi-pv^{p-1}\psi.
\end{align*}
Notice that the operator $\mathcal{A}_{0}:\psi\rightarrow-\lambda\Delta\psi+\psi$ is an isomorphism from $Y$ to $Y^{*}.$ And the operator $K:\psi\rightarrow pv^{p-1}\psi$ is compact from $Y$ to $Y^{*}$ since $v:=v_{k,\lambda}=\tilde{u}_\lambda\chi_{k}$ and $\tilde{u}_\lambda$ converges to $0$ at infinity exponentially. Since $\mathcal{A}=\mathcal{A}_{0}(I-\mathcal{A}^{-1}_{0}K)$ is injective and $\mathcal{A}^{-1}_{0}K$ is compact, the operator $\mathcal{A}$ is invertible by the Fredholm alternative for $I-\mathcal{A}^{-1}_{0}K$. This is a contradiction. Then there exists a non-zero function $\varphi_{k}\in Y$ such that
$$\tilde{F}'(v)\varphi_{k}=0.$$
Now, in the case $\|\tilde{F}'(v)^{-1}\|$ is unbounded, we assume that $\|\tilde{F}'(v)^{-1}\|\rightarrow\infty$ as $k \to 0$. Taking $\phi_k\in Y^{*}$ so that $\tilde{F}'(v)^{-1}\phi_k=\xi_{k}$ is a divergence sequence of $Y.$ Let $\varphi_{k}=\frac{\xi_{k}}{\|\xi_{k}\|_{k}},$ then $\|\varphi_{k}\|_{k}=1$ and define $\tilde{F}'(v)^{-1}\frac{\phi_k}{\|\xi_{k}\|_{k}}=\varphi_{k},$ one has that
$$\|\tilde{F}'(v)\varphi_{k}\|=\frac{\|\phi_k\|}{\|\xi_{k}\|_k}\rightarrow 0.$$
Therefore, it suffices to consider
$$\tilde{F}'(v)\varphi_{k}\rightarrow 0,$$
with $\|\varphi_{k}\|_k=1.$ First, we verify that $\varphi_{k}\rightharpoonup\varphi_{0}$ in $H_{r,loc}^{1}(\mathbb{R}^{d}\setminus B_{1})$ and $\varphi_{0}\in H_{0,r}^{1}(\mathbb{R}^{d}\setminus B_{1})$.

In fact, for $k = k(i,j)$ sufficiently small and a some fixed constant $M>1$, we use the Cantor's diagonal argument
\begin{equation*}
  \begin{matrix}
  \varphi_{k(1,1)}&\varphi_{k(1,2)}&\cdots&\varphi_{k(1,n)}\rightharpoonup\varphi_{0}^1 &\mbox{in $H^1_r(B_{M}\setminus B_{1})$ }\\
  \varphi_{k(2,1)}&\varphi_{k(2,2)}&\cdots&\varphi_{k(2,n)}\rightharpoonup\varphi_{0}^2 &\mbox{in $H^1_r(B_{2M}\setminus B_{1})$ }\\
  \vdots&\vdots&\cdots&\vdots&\vdots\\
  \varphi_{k(n,1)}&\varphi_{k(n,2)}&\cdots&\varphi_{k(n,n)}\rightharpoonup\varphi_{0}^n &\mbox{in $H^1_r(B_{nM}\setminus B_{1})$ }\\
   \vdots&\vdots&\cdots&\vdots&\vdots
  \end{matrix}
\end{equation*}
We consider the diagonal subsequence $\varphi_{k}=\varphi_{k(n,n)}$. Fix $m\in\mathbb{N},$ we can check that $(\varphi_{k(n,n)})_{n\geq m}\subset(\varphi_{k(m,n)})_{n\geq1}$ with $\varphi_{k(m,n)}\rightharpoonup\varphi_{0}^m$ in $H^1_r(B_{mM}\setminus B_{1})$ since each row is a subsequence of all previous rows. Then, by the uniqueness of limits, the function $\varphi_{0}=\varphi_{0}^m$ is well-defined. Thus we have $\varphi_{k}\rightharpoonup\varphi_{0}$ in $H_{r,loc}^{1}(\mathbb{R}^{d}\setminus B_{1}).$ And we know that
\begin{align*}
  \int_{1}^{\infty}r^{d-1}(\varphi'^{2}_{0}+\varphi^{2}_{0})dr&=\mathop {\lim}\limits_{M\rightarrow \infty}\int_{1}^{M}r^{d-1}(\varphi'^{2}_{0}+\varphi^{2}_{0})dr\\
     &\leq(1+\epsilon)^{d-1}\mathop {\lim}\limits_{M\rightarrow \infty}\mathop {\liminf}\limits_{k\rightarrow 0}\int_{1}^{M}\left(\frac{r}{1+\epsilon}\right)^{d-1}(\varphi'^{2}_{k}+\varphi^{2}_{k})dr\\
     &\leq(1+\epsilon)^{d-1}\mathop {\lim}\limits_{M\rightarrow \infty}\mathop {\liminf}\limits_{k\rightarrow 0}\int_{1}^{M}S^{d-1}_{k}(r)(\varphi'^{2}_{k}+\varphi^{2}_{k})dr\\
     &\leq(1+\epsilon)^{d-1}\mathop {\lim}\limits_{M\rightarrow \infty}\mathop {\liminf}\limits_{k\rightarrow 0}\int_{1}^{\frac{\pi}{k}}S^{d-1}_{k}(r)(\varphi'^{2}_{k}+\varphi^{2}_{k})dr\\
     &=(1+\epsilon)^{d-1},
\end{align*}
with small $\epsilon.$ This shows that $\varphi_{0}\in H_{0,r}^{1}(\mathbb{R}^{d}\setminus B_{1}).$

Now we prove that $\varphi_{0}$ differs from $0$ by contradiction. If $\varphi_{0}=0$,
 \begin{align*}
  \frac{\tilde{F}'(v)(\varphi_{k},\varphi_{k})}{\omega_{d-1}} &=\int_{1}^{\frac{\pi}{k}}S^{d-1}_{k}(r)(\lambda \varphi_{k}'^{2}+ \varphi_{k}^{2}-pv^{p-1}\varphi_{k}^{2})dr\\
     &\geq \min\{\lambda,1\}\|\varphi_{k}\|_k^{2}-p\int_{1}^{\frac{\pi}{k}}S^{d-1}_{k}(r)v^{p-1}\varphi_{k}^{2}dr\\
     &=\min\{\lambda,1\}-p\int_{1}^{\frac{\pi}{k}}S^{d-1}_{k}(r)v^{p-1}\varphi_{k}^{2}dr.
\end{align*}
Then we can get a contradiction that the left-hand side converges to $0$ but the right-hand side is greater and equal to $\min\{\lambda,1\}>0.$ The latter result comes from the fact that
\begin{align*}
   &\int_{1}^{\frac{\pi}{k}}S^{d-1}_{k}(r)v^{p-1}\varphi_{k}^{2}\\
   &=\int_{1}^{M}S^{d-1}_{k}(r)\tilde{u}_\lambda^{p-1}\varphi_{k}^{2}dr+\int_{M}^{\frac{\pi}{k}}S^{d-1}_{k}(r)\tilde{u}_\lambda^{p-1}\chi_{k}^{p-1}\varphi_{k}^{2}dr\\
   &=\int_{1}^{M}r^{d-1}\tilde{u}_\lambda^{p-1}\varphi_{0}^{2}dr+o_k(1)+\int_{M}^{\frac{\pi}{k}}S^{d-1}_{k}(r)\tilde{u}_\lambda^{p-1}\chi_{k}^{p-1}\varphi_{k}^{2}dr\\
   &\leq \int_{1}^{M}r^{d-1}\tilde{u}_\lambda^{p-1}\varphi_{0}^{2}dr+c|M|^{\frac{d-1}{2}}e^{\frac{-(p-1)M}{\sqrt{\lambda}}}\|\varphi_{k}\|_k^{2} + o_k(1),
\end{align*}
which can be taken arbitrarily small by choosing $M$ and $k$ appropriately.

Finally, for all $w\in C_{0}^{\infty}(B_{M}\setminus B_{1}),$
\begin{align*}
\frac{\tilde{F}'(v)(\varphi_{k},w)}{\omega_{d-1}}&=\int_{1}^{\frac{\pi}{k}}S^{d-1}_{k}(r)(\lambda \varphi_{k}'w'+ \varphi_{k}w-pv^{p-1}\varphi_{k}w)dr\\
&=\int_{1}^{M}S^{d-1}_{k}(r)(\lambda \varphi_{k}'w'+ \varphi_{k}w-pv^{p-1}\varphi_{k}w)dr\\
&\rightarrow \int_{1}^{M}r^{d-1}(\lambda \varphi_{0}'w'+ \varphi_{0}w-p\tilde{u}_\lambda^{p-1}\varphi_{0}w)dr\\
&= \int_{1}^{\infty}r^{d-1}(\lambda \varphi_{0}'w'+ \varphi_{0}w-p\tilde{u}_\lambda^{p-1}\varphi_{0}w)dr=0.
\end{align*}
As a consequence, $\varphi_{0} \neq 0$ solves the linearized problem
\begin{equation*}
\begin{cases}
-\lambda \varphi_{0}''-(d-1)\frac{\lambda}{r}\varphi_{0}'+\varphi_{0}-p\tilde{u}_\lambda^{p-1}\varphi_{0}=0 &\mbox{in $(1,\infty)$,}\\
\varphi_0(1)=0.
\end{cases}
\end{equation*}
But this is a contradiction with Proposition \ref{list}, c).

\medskip

\textbf{(A3)} Let's show that $\|\tilde{F}'(v)-F'(s)\|<\frac{1}{2c}$ for any $ s \in Y$ with $\|v-s\|\leq 2 c \,\delta.$ For any $\phi\in Y$, we have that
 \begin{align*}
 \left|\frac{\tilde{F}'(v)(\phi)w-\tilde{F}'(s)(\phi)w}{\omega_{d-1}}\right|&=p\left|\int_{1}^{\frac{\pi}{k}}S^{d-1}_{k}(r)
 (v^{p-1}-(s^{+})^{p-1})\phi wdr\right|\\
 &\leq p\int_{1}^{\frac{\pi}{k}}S^{d-1}_{k}(r)|\phi w|(\epsilon |s^{+}|^{p-1}+C_{\epsilon}|v-s^{+}|^{p-1})\\
&\leq C\|\phi\|_k\|w\|_k\|v-s^{+}\|_k^{p-1}\\
&\leq C(2c\delta)^{p-1}\|\phi\|_k\|w\|_k\\
&<\frac{1}{2c}\|\phi\|_k\|w\|_k,
\end{align*}
where $\epsilon$ is arbitrary and $C_{\epsilon}$ is a constant. This concludes the proof of the proposition.
\end{proof}

In next result we gather some properties of the boundedness and decay of the solution $u_{k, \lambda}$ that will be useful later.

\begin{lemma} \label{ultimo} Given $\e>0$, consider $k_0>0$, $k$ and $\lambda$ as in Proposition \ref{Pr21}. Then the following properties hold:
	\begin{enumerate}
		\item There exists $M>0$ independent of $k$, $\lambda$ such that $\| u_{k, \lambda} \|_{L^{\infty}} \leq M$.
		\item For any $\delta >0 $ there exists $R >0$ independent of $k$, $\lambda$ such that $u_{k, \lambda}(r) < \delta$ for any $k < \pi/R$, $r \in (R, \pi/k)$.
\end{enumerate}	\end{lemma}

\begin{proof}
	
The proof of 1) follows immediately from the classical blow-up analysis of Gidas and Spruck, together with their Liouville theorem, see \cite{gs1, gs2}. Observe that here we use in an essential way that $1<p< \frac{d+2}{d-2}$ (if $d \geq 3$).

We now turn our attention to 2). Observe that by Proposition \ref{Pr21}, $\| u_{k, \lambda} \|_k \leq C$. In particular, taking into account that $S_k(r) \leq 2r /\pi $ for all $r \in (1, \frac{\pi}{2k}]$,

$$ \int_1^{\frac{\pi}{2k}} r^{d-1} (u_{k, \lambda}'(r)^2 + u_{k, \lambda}(r)^2) \, dr \leq C.$$

Then, the Radial Lemma of Strauss \cite{strauss} gives us the desired decay for any $r \in (1, \frac{\pi}{2k}]$.

We now consider the estimate in the north hemisphere $\mathbb{S}_+^{d}(k)$. If it does not hold, then $u_{k,\lambda}$ attains a local maximum in such hemisphere for a sequence $k=k_n \to 0$. By the maximum principle, this local maximum is strictly bigger than 1.

We now multiply \eqref{eq13} by $v= (u_{k, \lambda}-1)^+$ and integrate in $\mathbb{S}_+^{d}(k)$, to obtain that:

$$ \lambda \int_{\mathbb{S}_+^{d}(k)}  |\nabla v|^2 =  \int_{\mathbb{S}_+^{d}(k)} \frac{u_{k, \lambda}^p - u_{k, \lambda}}{u_{k, \lambda}-1} v^2 \leq C \int_{\mathbb{S}_+^{d}(k)} v^2,$$
by 1). We now define $\Sigma=\{x \in \mathbb{S}_+^{d}(k): \ u_{k, \lambda}(x) >1 \}$. Therefore, we can use H\"{o}lder inequality for any $2< q \le\frac{2d}{d-2}$ (any $q >2$ if $d=2$), and Lemma \ref{Le202}, to obtain:

\begin{align*} & \lambda \int_{\mathbb{S}_+^{d}(k)}  (|\nabla v|^2 +v^2) \leq (C+ \lambda) \int_{\mathbb{S}_+^{d}(k)} v^2 \\ &\leq (C+ \lambda) \Big( \int_{\mathbb{S}_+^{d}(k)} v^q \Big)^{2/q} |\Sigma| ^{\frac{q-2}{q}} \leq C \Big( \int_{\mathbb{S}_+^{d}(k)}  (|\nabla v|^2 +v^2) \Big)|\Sigma| ^{\frac{q-2}{q}}. \end{align*}

 This implies that the measure of $\Sigma$ is uniformly bounded from below. Observe moreover that:

$$ \int_{\mathbb{S}_+^{d}(k)} u_{k, \lambda}(x)^2 \, dx \geq |\Sigma|,$$
which is bounded from below. But this is in contradiction with \eqref{limitu} and the exponential decay of $\tilde{u}_{\lambda}$.

\end{proof}

\section{The Dirichlet-to-Neumann operator and its linearization}
\label{Section 4}

In this section we build the Dirichlet-to-Neumann operator to which we intend to apply a local bifurcation argument. In order to do this, some definitions are in order. First, let us fix a symmetry group $G$ satisfying assumption (G).

For any solution given in Proposition \ref{Pr21}, we define the linearized operator of the Dirichlet problem: $L^D=L_{k, \lambda}^D:H^{1}_{0,G}(\mathbb{S}^{d}(k)\setminus B_{1})\rightarrow H^{*}_G(\mathbb{S}^{d}(k)\setminus B_{1})$ associated to (\ref{eq13}) by
\begin{equation}\label{eq220}
	L^D(\phi)=-\lambda\Delta\phi+\phi-pu_{k, \lambda}^{p-1}\phi,
\end{equation}
where $H^{*}_G(\mathbb{S}^{d}(k)\setminus B_{1})$ is the dual of $H^{1}_{0,G}(\mathbb{S}^{d}(k)\setminus B_{1}).$

\medskip

\begin{proposition} \label{index} The operator $L^D$ has a negative eigenvalue $\tau=\tau_{k, \lambda}$ with a radially symmetric eigenfunction $z= z_{k, \lambda}$. \end{proposition}

The proof is rather easy and will be presented in Section 5, when we introduce quadratic forms associated to $L^D$.

\medskip We also need the following result, that establishes the nondegeneracy of the Dirichlet operator:

\begin{proposition} \label{nondeg D} There exist $0< \lambda_0 < \lambda_1$ such that, by taking smaller $k_0 >0$, if necessary, we have that for any $k \in (0, k_0)$, $\lambda \in [\lambda_0, \lambda_1]$, the eigenvalues of $L^D$ different from $\tau$ are all strictly bigger than a positive constant $\e$ independent of $\lambda$, $k$. In particular, the operator $L^D$ is an isomorphism.
	
\end{proposition}

The proof of this proposition is an immediate consequence of Proposition \ref{Pr25}, iii). We emphasize that the values $\lambda_0$, $\lambda_1$, which will be fixed in the rest of the section, are given in Proposition \ref{Pr26} and depend only on $d$, $p$ and $G$.

\medskip The main result of this section is the following:

\begin{proposition} \label{Pr30}
	Assume that $\lambda\in[\lambda_{0},\lambda_{1}]$, and $k \in (0, k_0)$, where $k_0$ is given in Proposition \ref{nondeg D}. Then, there exists a neighborhood $\mathcal{U}$ of $0$ in $ C_{G,0}^{2,\alpha}(\mathbb{S}^{d-1})$, independent of $\lambda$, $k$, such that for any $v \in \mathcal{U}$, the problem
	\begin{equation}\label{eq31}
	\begin{cases}
	-\lambda\Delta u+u-u^{p} =0&\emph{in $\mathbb{S}^{d}(k)\setminus B_{1+v}$},\\
	u>0 & \emph{in $\mathbb{S}^{d}(k)\setminus B_{1+v}$}, \\
	u=0 &\emph{on $\partial B_{1+v}$},
	\end{cases}
	\end{equation}
	has a unique positive solution $u=u_{k, \lambda}^v\in C^{2,\alpha}(\mathbb{S}^{d}(k)\setminus B_{1+v})$ in a neighborhood of $u_{k, \lambda}$.  %
	Moreover the dependence of $u$ on the function $v$ is $C^1$ and $u_{k, \lambda}^0=u_{k, \lambda}$.
\end{proposition}
\begin{proof}
 Let $v\in C_{G}^{2,\alpha}(\mathbb{S}^{d-1})$. It will be more convenient to consider the fixed domain $\mathbb{S}^{d}(k)\setminus B_{1}$ endowed with a new metric depending on $v$. This will be possible by considering the parameterization of $\mathbb{S}^{d}(k)\setminus B_{1+v}$ defined by $\Xi:\mathbb{S}^{d}(k)\setminus B_{1}\rightarrow\mathbb{S}^{d}(k)\setminus B_{1+v},$
	\begin{equation}\label{eq32}
	\Xi(r,\theta):=\left(\left(1+\chi(r)v(\theta)\right)r,\theta\right),
	\end{equation}
	where $\chi$ is a cut-off function
	\begin{equation*}
	 \chi(r)=
	\begin{cases}
	0, &r \geq\frac{3}{2},\\
	1, &r \leq\frac{5}{4}.
	\end{cases}
	\end{equation*}	
	Therefore, we consider the coordinates $(r,\theta)\in \left(1,\frac{\pi}{k}\right)\times\mathbb{S}^{d-1}$ from now on, and we can write the new metric in these coordinates as
\[g_{v}=a^{2}dr^{2}+2abdrdv+b^{2}dv^{2}+S^{2}_{k}\left(\left(1+\chi(r)v(\theta)\right)r\right)\, \mathring{h} ,\]
where $a=1+\chi(r)v(\theta)+\chi'(r)v(\theta)r$, $b=\chi(r)r$, and $\mathring{h}$ is the standard metric on $\mathbb{S}^{d-1}$ induced by the Euclidean one. It is clear that $a=1,b\equiv 0$ if $v=0$. Up to some multiplicative constant, we can now write the problem (\ref{eq31}) as
	\begin{equation}\label{eq33}
	\begin{cases}
-\lambda\Delta_{g_{v}} \hat{u}+\hat{u}-\hat{u}^{p}=0 &\mbox{in $\mathbb{S}^{d}(k)\setminus B_{1}$ },\\
	\hat{u}=0 &\mbox{on $\partial B_{1}$}.
	\end{cases}
	\end{equation}
	As $v\equiv0$, the metric $g_{v}$ is just the round metric $g_{k}$, and $\hat{u}=u_{k,  \lambda}$ is therefore a solution of (\ref{eq33}).  In the general case, the expression between the function $u$ and the function $\hat{u}$ can be represented by
	\[\hat{u}=\Xi^{*}u.\]
	For all $\psi\in H_{0,G}^{1}(\mathbb{S}^{d}(k)\setminus B_{1}),$ we define:
	\begin{equation} \label{eq34} N(v,\psi,\lambda):=-\lambda\Delta_{g_{v}}(u_{k,\lambda}+\psi)+(u_{k, \lambda}+\psi)-[(u_{k, \lambda}+\psi)^{+}]^{p},\end{equation}
	where $(u_{k, \lambda}+\psi)^{+}$ is the positive part of $u_{k,\lambda}+\psi.$ We have\[N(0,0,\lambda)=0.\]
	The mapping $N$ is $C^{1}$ from a neighborhood of $(0,0,\lambda)$ in $C_{G}^{2,\alpha}(\mathbb{S}^{d-1})\times H_{0,G}^{1}(\mathbb{S}^{d}(k)\setminus B_{1})\times[\lambda_{0},\lambda_{1}]$ into $H^{*}_G(\mathbb{S}^{d}(k)\setminus B_{1}).$ The partial differential of $N$ with respect to $\psi$ at $(0,0,\lambda)$ is
	\[D_{\psi}N|_{(0,0,\lambda)}(\psi)=-\lambda\Delta\psi+\psi-pu_{k,\lambda}^{p-1}\psi.\]
Since $\lambda\in[\lambda_{0},\lambda_{1}]$, $D_{\psi}N|_{(0,0,\lambda)}(\psi)$ is precisely invertible from $H_{0,G}^{1}(\mathbb{S}^{d}(k)\setminus B_{1})$ into $H^{*}_G(\mathbb{S}^{d}(k)\setminus B_{1})$ by the fact that the operator $L_{k, \lambda}^D$ is nondegenerate, see Proposition \ref{nondeg D}. The Implicit Function Theorem therefore yields that there exists $\psi(v,\lambda)\in H_{0,G}^{1}(\mathbb{S}^{d}(k)\setminus B_{1})$ such that $N(v,\psi(v,\lambda),\lambda)=0$ for $v$ in a neighborhood of $0$ in $C_{G}^{2,\alpha}(\mathbb{S}^{d-1})$. Observe that the neighborhood $\mathcal{U}$ can be taken uniformly in $k$ by the quantitative version of the Implicit Function Theorem, see Proposition \ref{Pr22}.
The function $\hat{u}:=u_{k,\lambda}+\psi$ so obtained solves the problem $-\lambda\Delta_{g_{v}} \hat{u}+\hat{u}-(\hat{u}^{+})^{p}=0.$ By the maximum principle, $\hat{u}$ is positive and solves \eqref{eq33}.
\end{proof}

Let $k \in (0, k_0)$. After the canonical identification of $\partial B_{1+v}$ with $\mathbb{S}^{d-1},$ we define $F_k: \mathcal{U}\times[\lambda_{0},\lambda_{1}] \rightarrow C^{1,\alpha}_{G,0}(\mathbb{S}^{d-1})$,
\begin{equation}\label{eq35}
F_k(v,\lambda)= \partial_{\nu} u \Big|_{\partial B_{1+v}}-\frac{1}{\mbox{Vol}(\partial B_{1+v})}\int_{\partial B_{1+v}}\partial_{\nu} u.
\end{equation}
Here $\mathcal{U}$ and $u=u_{k, \lambda}^v$ are as given by Proposition \ref{Pr30}. Notice that $F_k(v,\lambda)=0$ if and only if $\partial_{\nu} u$ is constant on the boundary $\partial B_{1+v}.$ Obviously, $F_k(0,\lambda)=0$ for all $\lambda\in[\lambda_{0},\lambda_{1}].$ Our goal is to find a branch of nontrivial solutions $(v,\lambda)$ to the equation $F_k(v,\lambda)=0$ bifurcating from some point $(0, \lambda_*(k))$, $\lambda_*(k) \in[\lambda_{0},\lambda_{1}].$ For this aim, we will use a local bifurcation argument. This leads to the study of the linearization of $F_k$ around a point $(0,\lambda)$. To start that, we first show the following useful lemmas.

 \begin{lemma} \label{le29}
Assume that $\lambda\in[\lambda_{0},\lambda_{1}]$, and $k \in (0, k_0)$, where $k_0$ is given in Proposition \ref{nondeg D}. Then for all $v\in C_{G}^{2,\alpha}(\mathbb{S}^{d-1})$, there exists a unique solution $\psi= \psi_{k, \lambda}^v \in C^{2, \alpha}_G(\mathbb{S}^{d}(k)\setminus B_{1})$ to the problem
\begin{equation}\label{eq23}
  \begin{cases}
  -\lambda\Delta\psi +\psi -p \, u_{k,\lambda}^{p-1}\psi =0 &\emph{in $\mathbb{S}^{d}(k)\setminus B_{1}$, }\\
  \psi= v &\emph{on $\partial B_1$.}
  \end{cases}
\end{equation}
\end{lemma}
\begin{proof} Let $\psi_{0}(x) \in C_{G}^{2, \alpha}(\mathbb{S}^{d}(k)\setminus B_{1})$ such that $\psi_0|_{\partial B_1} = v$.
If we set $w=\psi-\psi_{0}$, the problem (\ref{eq23}) is equivalent to the problem
\begin{equation*}
  \begin{cases}
  -\lambda\Delta w+w-p u_{k,\lambda}^{p-1}w= - \Big ( -\lambda\Delta \psi_{0}+\psi_{0}-pu_{k,\lambda}^{p-1}\psi_{0}  \Big ) &\mbox{in $\mathbb{S}^{d}(k)\setminus B_{1}$, }\\
 w=0 &\mbox{on $\partial B_1$.}
  \end{cases}
\end{equation*}
Observe that the right hand side of the above equation is in $H^{*}_G(\mathbb{S}^{d}(k)\setminus B_{1})$. Since by Proposition \ref{nondeg D} the operator $L_{k, \lambda}^D$ is a bijection, there exists a solution $w$. By Schauder estimates, $w$ has the required regularity and the result follows.
\end{proof}
\begin{lemma} \label{Le31} Assume that $\lambda\in[\lambda_{0},\lambda_{1}]$, and $k \in (0, k_0)$, where $k_0$ is given in Proposition \ref{nondeg D}.
Let $v\in C_{G,0}^{2,\alpha}(\mathbb{S}^{d-1})$ and $\psi \in C_{G}^{2,\alpha}(\mathbb{S}^{d}(k)\setminus B_{1})$ be the solution of  (\ref{eq23}). Then
 \[\int_{\mathbb{S}^{d}(k)\setminus B_{1}}\psi z=0,~\qquad  \int_{\partial B_1}\partial_{\nu}\psi =0.\]
 Here $z$ stands for the eigenfunction associated to the negative eigenvalue of $L^D$, as given in Proposition \ref{index}.
\end{lemma}

\begin{proof} {Let
\[L^{D}(z)=\tau z \,\ \mbox{in $\mathbb{S}^{d}(k)\setminus B_{1}$}, \,\, z=0 \,\, \mbox{on $\partial B_1$},\]
where $L^{D}$ is given in (\ref{eq220}). We now multiply the above equation by $\psi=\psi_{k, \lambda}^v ,$} the equation in (\ref{eq23}) by $z$, and
integrate by parts to gain
\[\int_{\partial B_1}\Big(\partial_{\nu} \psi \, z -\partial_{\nu} z \, \psi \Big)=\int_{\mathbb{S}^{d}(k)\setminus B_{1}}\tau z \psi \,.\]
Then we can at once gain the first identity by the facts that $z=0$, $\partial_{\nu} z$ is constant and $\psi= v$ has $0$ mean on $\partial B_1$.

\medskip

We now define $\vartheta\in H_{G}^{1}(\mathbb{S}^{d}(k)\setminus B_{1})$ as the unique solution of the problem
\begin{equation}\label{eq355}
     \begin{cases}
     -\lambda\Delta\vartheta+\vartheta-pu_{k,\lambda}^{p-1}\vartheta=0 &\mbox{in $\mathbb{S}^{d}(k)\setminus B_{1}$, }\\
      \vartheta=1 &\mbox{on $\partial B_1$,}
      \end{cases}
    \end{equation}
whose existence has been proved in Lemma ~\ref{le29}. Observe that, by uniqueness, $\vartheta$ is radially symmetric. Then we multiply the equation in \eqref{eq355} by $\psi$, the equation in \eqref{eq23} by $\vartheta$, and integrate by parts to obtain
\[\int_{\partial B_{1}}\Big(\partial_{\nu} \psi \, \vartheta-\partial_{\nu} \vartheta \,  \psi \Big)=0.\]
Then we can immediately gain the second identity by the facts that $\vartheta=1, \partial_{\nu} \vartheta$ is constant and $\psi= v$ on $\partial B_1$.
\end{proof}

\medskip

For $\lambda\in[\lambda_{0},\lambda_{1}]$, $k \in (0, k_0)$, we can define the linear continuous operator $H_{k,\lambda}:C^{2,\alpha}_{G,0}(\mathbb{S}^{d-1})\rightarrow C_{G,0}^{1,\alpha}(\mathbb{S}^{d-1})$ by
\begin{equation}\label{eq36}
H_{k, \lambda}(v) =  \partial_{\nu} (\psi_v) -\frac{(d-1)k}{\tan(k)} \, v,
\end{equation}
where $\psi_v$ is given by Lemma ~\ref{le29}. It is worth pointing out that the constant $$\frac{(d-1)k}{\tan(k)}$$ is nothing but the mean curvature of $\partial B_1 \subset \Sk$.

\medskip

We show now that the linearization of the operator $F_k$ with respect to $v$ at $v=0$ is given by $H_{k,\lambda}$, up to a constant.

\begin{proposition} \label{Pr32}
For any $\lambda\in[\lambda_{0},\lambda_{1}]$ and $k \in (0, k_0)$ we have
\[
D_{v}(F_k)(\lambda, 0) = \partial_{r}u(1)\, H_{k, \lambda}\,,
\]
where $u=u_{k,\lambda}$.
 \end{proposition}

\begin{proof}
By the $C^1$ regularity of $F_k$, it is enough to compute the linear operator obtained by the directional derivative of $F_k$ with respect to $v$, computed at $(0,\lambda)$. Such derivative is given by
\begin{align*}
   F_k'(w)=\mathop {\lim}\limits_{s\rightarrow 0}\frac{F_k(sw,\lambda)-F_k(0,\lambda)}{s}=\mathop {\lim}\limits_{s\rightarrow 0}\frac{F_k(sw,\lambda)}{s}.
    \end{align*}
    Let $v=sw,$ we consider the parameterization of $\mathbb{S}^{d}(k)\setminus B_{1+v}$ given in (\ref{eq32}) for $(r,\theta)\in \left(1,\frac{\pi}{k}\right)\times\mathbb{S}^{d-1}$. Let $g_{v}$ be the induced metric such that $\hat{u} = \hat u^v_{k,\lambda} =\Xi^{*}u^v_{k,\lambda}$ (smoothly depending on the real parameter $s$) solves the problem
    \begin{equation*}
     \begin{cases}
     -\lambda\Delta_{g_{v}}\hat{u}+\hat{u}-\hat{u}^{p}=0 &\mbox{in $\mathbb{S}^{d}(k)\setminus B_{1}$, }\\
      \hat{u}=0 &\mbox{on $\partial B_{1}$}.
      \end{cases}
    \end{equation*}
 We define $\hat{u}_{k, \lambda}=\Xi^{*}u_{k,\lambda}$, which is a solution of
\[-\lambda\Delta_{g_{v}}\hat{u}_{k,\lambda}+\hat{u}_{k,\lambda}-\hat{u}_{k,\lambda}^{p}=0\]
in $\mathbb{S}^{d}(k)\setminus B_{1}$ (notice that $u_{k,\lambda}$ is radial and then can be extended as a solution of \eqref{eq14} in a small inner neighborhood of $\partial B_1$), and
\[\hat{u}_{k,\lambda}(r,\theta)=u_{k,\lambda}\big((1+sw)r,\theta\big)\]
on $\partial B_1$. Let $\hat{u}=\hat{u}_{k,\lambda}+\hat{\psi},$ we can get that
\begin{equation}\label{eq37}
     \begin{cases}
     -\lambda\Delta_{g_{v}}\hat{\psi}+(\hat{u}_{k,\lambda}+\hat{\psi})-(\hat{u}_{k,\lambda}+\hat{\psi})^{p}-\hat{u}_{k,\lambda}+\hat{u}_{k,\lambda}^{p}=0  &\mbox{in $\mathbb{S}^{d}(k)\setminus B_{1}$ },\\
      \hat{\psi}=-\hat{u}_{k,\lambda}  &\mbox{on $\partial B_1$}.
      \end{cases}
    \end{equation}
 Obviously, $\hat{\psi}$ is differentiable with respect to $s$. When $s=0$, we have $\hat{u}=u_{k,\lambda}.$ Then, $\hat{\psi}=0$ as $s=0.$ We set
  \[\dot{\psi}=\partial_{s}\hat{\psi}|_{s=0}.\]
 Differentiating (\ref{eq37}) with respect to $s$ and evaluating the result at $s=0$, we get that
  \begin{align*}
     \begin{cases}
     -\lambda\Delta\dot{\psi}+\dot{\psi}-pu_{k,\lambda}^{p-1}\dot{\psi}=0  &\mbox{in $\mathbb{S}^{d}(k)\setminus B_{1}$, }\\
      \dot{\psi}=-\partial_{r}u_{k,\lambda}(1) w  &\mbox{on $\partial B_1$.}
  \end{cases}
    \end{align*}
  Then $\dot{\psi} = -\partial_{r}u_{k,\lambda}(1)\, \psi$ where $\psi=\psi_{k, \lambda}^v$ is as given in Lemma ~\ref{le29} (with $v=w$). Then, we can write
     \[\hat{u}(r,\theta)=\hat{u}_{k,\lambda}(r,\theta)-s\partial_{r}u_{k,\lambda}(1)\, \psi(r,\theta)+o(s).\]
    In particular, in $B_{5/4}\setminus B_{1}$ we have
   \begin{align*}
   \hat{u}(r,\theta)&=u_{k,\lambda}((1+sw(\theta))r,\theta)-s\partial_{r}u_{k,\lambda}(1)\, \psi(r,\theta)+o(s)\\
         &=u_{k,\lambda}(r,\theta)+s\big(rw(\theta)\partial_{r}u_{k,\lambda}(r,\theta)-\partial_{r}u_{k,\lambda}(1)\, \psi(r,\theta)\big)+o(s).
         \end{align*}
In order to complete the proof of the result, it is enough to calculate the normal derivation of the function $\hat{u}$ when the normal is calculated with respect to the metric $g_{v}$. Since the coordinates $(r,\theta)\in \left(1,\frac{\pi}{k}\right)\times \mathbb{S}^{d-1}$, the metric $g_{v}$ can be expanded in $B_{5/4}\setminus B_{1}$ as
         \[g_{v}=(1+sw)^{2}dr^{2}+2sr(1+sw)drdw+s^{2}r^{2}dw^{2}+S_{k}^{2}((1+sw)r)\, \mathring{h}\,,\]
         where again $\mathring{h}$ is the standard metric on $\mathbb{S}^{d-1}$ induced by the Euclidean one.
It follows from this expression that the unit normal vector field to $\partial B_1$ for the metric $g_{v}$ is given by
          \[\hat{\nu}=\big((1+sw)^{-1}+o(s)\big)\partial_{r}+o(s)\partial_{\theta_{i}},\]
          where $\theta_{i}$ are the vector fields induced by a parameterization of $\mathbb{S}^{d-1}$. As a result,
   \[g_{v}(\nabla\hat{u},\hat{\nu})=\partial_{r}u_{k,\lambda}+s\big(w\partial_{r}^{2}u_{k,\lambda}-\partial_{r}u_{k,\lambda}\partial_{r}\psi\big)+o(s)\]
   on $\partial B_1$. From the fact that $\partial_{r}u_{k,\lambda}$ and $\partial_{r}^{2}u_{k,\lambda}$ are constant, while the term $w$ and $\psi$ have mean $0$ on $\partial B_1,$ and
   \[-\lambda\left(\partial_{r}^{2}u_{k,\lambda}+\frac{(d-1)k}{\tan(k)}\partial_{r}u_{k,\lambda}\right)=0\] on $\partial B_1$,
   we can conclude the proof of the result by using Lemma \ref{Le31}.
  \end{proof}

\section{Study of the linearized operator $H_{k, \lambda}$}
\label{Section 5}

In view of Proposition \ref{Pr32}, a bifurcation of the branch $(0, \lambda)$ of solutions to the equation $F_k(v, \lambda)=0$ might appear only at point $(0,\lambda_{*})$ such that $H_{k, \lambda_{*}}$ becomes degenerate. We will see that this is true for some $\lambda_{*}$.


\medskip

As we shall see in Lemma \ref{le43}, the behavior of $H_{k, \lambda}$ is related to the following quadratic form:
\begin{equation} \label{Q}
\begin{array}{c} Q_{k,\lambda}: H^1_{G}(\Sk \setminus B_1) \to \R, \\ \\ 	Q_{k,\lambda}(\psi)= \displaystyle \int_{\mathbb{S}^{d}(k)\setminus B_{1}}\big(\lambda |\nabla \psi|^2 +\psi^2-p u_{k, \lambda}^{p-1}\psi^2 \big) -\lambda\frac{(d-1)k}{\tan(k)}\int_{\partial B_{1}} \psi^2. \end{array}
\end{equation}
When restricted to functions that vanish at the boundary, we obtain the quadratic form associated to the Dirichlet problem:
\begin{equation} \label{QD}
\displaystyle	Q^D_{k,\lambda} := Q_{k,\lambda}|_{ H^1_{0,G}(\Sk \setminus B_1) }.
\end{equation}
Sometimes we will simply write $Q^D$ instead of $Q^D_{k,\lambda}$.

\medskip

We first show that $Q^D$ attains negative values, which shows the validity of Proposition \ref{index}.

\begin{proof}[Proof of Proposition \ref{index}]
 In this proof we drop the subindices $k, \lambda$. Observe that by multiplying equation \eqref{eq13} by $u$ and integrating by parts we obtain that:
$$ \int_{\mathbb{S}^{d}(k)\setminus B_{1}} \lambda |\nabla u|^{2}+ |u|^{2}-u^{p+1}=0.$$
As a consequence,
$$ Q^{D}(u) = -(p-1) \int_{\mathbb{S}^{d}(k)\setminus B_{1}} u^{p+1}<0.$$
Then, the first eigenvalue of $L^D$ is strictly  negative. Since the first eigenfunction is simple, it is radially symmetric.
\end{proof}

In what follows it will be necessary to restrict those quadratic forms to the following spaces:
 \begin{equation}\label{E}
	\begin{array}{c}	E_{k, \lambda}=\Bigg\{\psi\in H^1_G(\mathbb{S}^{d}(k)\setminus B_{1}): \ \dis\int_{\partial B_1}\psi=0,\int_{\mathbb{S}^{d}(k)\setminus B_{1}}\psi z_{k,\lambda}=0\Bigg\}, \\ \\
		E_{k, \lambda}^D=\Bigg\{\psi\in H^1_{0,G}(\mathbb{S}^{d}(k)\setminus B_{1}): \ \dis \int_{\mathbb{S}^{d}(k)\setminus B_{1}}\psi z_{k,\lambda}=0\Bigg\}. \end{array}
\end{equation}
As $k$ tends to $0$ we have formally limit quadratic forms related to the problem in $\R^d \setminus B_1$. On that purpose, let us define:
\begin{equation} \begin{array}{c}\label{limitQ} \displaystyle  \tilde{Q}_{\lambda}: H_{G}^1(\mathbb{R}^{d}\setminus B_{1})\rightarrow\mathbb{R}  \\   \\ \displaystyle
		\tilde{Q}_{\lambda}(\psi):=\int_{\mathbb{R}^{d}\setminus B_{1}} \lambda |\nabla \psi|^{2}+ \psi^{2}-p\tilde{u}_{\lambda}^{p-1}\psi^{2} - \lambda (d-1)  \int_{\partial B_{1}}\psi^2, \\ \\ \displaystyle \tilde{Q}^D_{\lambda} = \tilde{Q}_{\lambda}|_{ H^1_{0,G}(\R^d \setminus B_1) }. \end{array}
\end{equation}
Here $\tilde{u}_{\lambda}$ is the solution given in Proposition \ref{list}. We also define the analogous functional spaces:
 \begin{equation}\label{limitE}
  \begin{array}{c}	\tilde{E}_{\lambda}=\Bigg\{\psi\in H^1_G(\mathbb{R}^{d}\setminus B_{1}): \ \dis\int_{\partial B_1}\psi=0,\int_{\mathbb{R}^{d}\setminus B_{1}}\psi \tilde{z}_{\lambda}=0\Bigg\}, \\ \\
  	\tilde{E}^D_{\lambda}=\Bigg\{\psi\in H^1_{0,G}(\mathbb{R}^{d}\setminus B_{1}): \ \dis \int_{\mathbb{R}^{d}\setminus B_{1}}\psi \tilde{z}_{\lambda}=0\Bigg\}. \end{array}
\end{equation}
In order to facilitate the reading, sometimes we will drop the subscripts $k, \lambda$ of the above definitions.

\medskip

The behavior of the quadratic forms $\tilde{Q}_{\lambda}$ and $\tilde{Q}^D_{\lambda}$ has been studied in \cite{RRS20}, and the following result holds true.

\begin{proposition} \label{Pr26} There exist $0 < \lambda_0 < \lambda_1$ and $\e > 0$ such that:
	\begin{itemize}
		\item [(i)] $\tilde{Q}_{\lambda_{0}}(\psi)< - \e \ $ for some $\psi\in \tilde{E}_{\lambda_0}$, $\| \psi \|_{L^2} =1$;
		\item [(ii)] $\tilde{Q}_{\lambda_1}(\psi)> \e \ $  for any $\psi\in \tilde{E}_{\lambda_1}$, $\| \psi \|_{L^2} =1$;
		\item [(iii)] $\tilde{Q}^D_{\lambda}(\psi)> \e \ $ for any $\psi\in \tilde{E}^D_{\lambda}$ , $\| \psi \|_{L^2} =1$, and any $\lambda > \lambda_{0}.$
	\end{itemize}
\end{proposition}

\begin{proof} The proof is basically contained in \cite{RRS20}. We first recall the definitions given in \cite[(3.6) and (5.7)]{RRS20} under our notations:
	
\begin{equation} \label{Lambda0}  \Lambda_0 = \sup \left\{ \lambda >0: \tilde{Q}_\lambda^D(\psi) \leq 0 \mbox{ for some } \psi \in \tilde{E}_\lambda^D, \ \psi \neq 0 \right\}. \end{equation}

\begin{equation} \label{Lambda*} \Lambda^* = \sup \{ \lambda >0: \tilde{Q}_\lambda(\psi) < 0 \ \mbox{ for some } \psi \in \tilde{E}_{\lambda}\}. \end{equation}

It is proved in \cite{RRS20} that the above suprema exist and that $0 < \Lambda_0 < \Lambda^*$. Then, we can take $\lambda_0 \in  (\Lambda_0, \Lambda^*)$ such that i) and iii) hold.

\medskip

Moreover, \cite[Proposition 5.3]{RRS20} implies that $\tilde{Q}(\psi)>0$ for any $\psi \in \tilde{E}_\lambda$, $\psi \neq 0$, provided that $\lambda$ is sufficiently large. Then we can take $\lambda_1 > \Lambda^*$ such that ii) is satisfied.

\end{proof}

In next proposition we use a perturbation argument to prove an analogous result for small $k>0$.
\begin{proposition} \label{Pr25} Fix $\lambda_0$ and $\lambda_1$ as given in Proposition \ref{Pr26}. By taking $k_0>0$ smaller if necessary, and for any $k \in (0, k_0)$ there exists $\e>0$ independent of $k$, $\lambda$ such that:
\begin{itemize}
\item [(i)] $Q_{k, \lambda_{0}}(\psi)< -\e $ for some $\psi\in E_{k, \lambda_0}$, $\| \psi \|_{L^2} =1$:
 \item [(ii)] $Q_{k, \lambda_{1}}(\psi)> \e$ for any $\psi\in E_{k, \lambda_1}$,  $\| \psi \|_{L^2} =1$:
 \item [(iii)] $Q_{k, \lambda}^D(\psi)> \e$ for any $\psi\in E^D_{k, \lambda}$, $\| \psi \|_{L^2} =1$ and any $\lambda\in[\lambda_{0},\lambda_{1}]$.

  \end{itemize}

In particular, (iii) implies Proposition \ref{nondeg D}.
\end{proposition}
In order to prove the above Proposition, we need the following lemma, whose proof will be given in the Appendix.
 \begin{lemma} \label{Le32} For any function $\psi\in H^{1}(\mathbb{S}^{d}(k)\setminus B_{1}),$ we have that
  \begin{align*}
 \|\psi\|^{2}_{L^{2}(\partial B_{1})}\leq 2\|\nabla \psi\|_{L^{2}(\mathbb{S}^{d}(k)\setminus B_{1})}\|\psi\|_{L^{2}(\mathbb{S}^{d}(k)\setminus B_{1})}+C\|\psi\|^{2}_{L^{2}(\mathbb{S}^{d}(k)\setminus B_{1})}
\end{align*}
for some constant $C=C(d)>0$, which does not depend on $k$.
 \end{lemma}

\medskip

\begin{proof}[Proof of Proposition \ref{Pr25}.]
	
First, we prove that (i) holds. Let us give the following min-max characterization of the second eigenvalue related to the quadratic forms $Q_{k, \lambda}$ and $\tilde{Q}_{\lambda}$ as follows. Define $\mathcal{A}_{k}$ and $\tilde{\mathcal{A}}$ the class of 2-dimensional vector spaces in
\[\left\{\psi\in H^1_G(\mathbb{S}^{d}(k)\setminus B_{1}): \ \dis\int_{\partial B_1}\psi=0\right\},\]
and
\[\left\{\psi\in H^1_G(\mathbb{R}^{d}\setminus B_{1}): \ \dis\int_{\partial B_1}\psi=0\right\},\]
respectively. Then we are concerned with the infimum:

\begin{align*}
 \inf\limits_{U}\{\max\{Q_{k, \lambda}(\psi):\psi\in U,\|\psi\|_{L^{2}(\mathbb{S}^{d}(k)\setminus B_{1})}=1\}:U\in\mathcal{A}_{k}\},
\end{align*} and
\begin{align*}
 \inf\limits_{\tilde{U}}\{\max\{\tilde{Q}_{\lambda}(\psi):\psi\in \tilde{U},\|\psi\|_{L^{2}(\mathbb{R}^{d}\setminus B_{1})}=1\}:U\in\tilde{\mathcal{A} } \}.
\end{align*}
Observe that the last infimum is strictly negative for $\lambda= \lambda_0$ by Proposition \ref{Pr26}, i).
By density, there exists a 2-dimensional vector space of $U$ of functions in $C^{\infty}_{G}(\mathbb{R}^{d}\setminus B_{1})$ with support contained in a fixed compact set $K$ such that
\[\max\{\tilde{Q}_{\lambda_{0}}(\psi):\psi\in \tilde{U},\|\psi\|_{L^{2}(\mathbb{R}^{d}\setminus B_{1})}=1\}<0.\]
We also have
\[\max\{\tilde{Q}_{\lambda_{0}}(\psi):\psi\in \tilde{U},\|\psi\|_{L^{2}(\mathbb{S}^{d}(k)\setminus B_{1})}=1\}<0\]
for $k$ small, since both norms are equivalent in the 2-dimensional vector space $\tilde{U}.$

 Therefore, for any $\psi\in \tilde{U}$ with $\|\psi\|_{L^{2}(\mathbb{S}^{d}(k)\setminus B_{1})}=1$, one has that
 \begin{align*}
 Q_{k,\lambda_{0}}(\psi) =\tilde{Q}_{\lambda_{0}}(\psi) + o_k(1) <0.
\end{align*}
Clearly, we can choose such $\psi$ orthogonal to $z_{k,\lambda}$, so that \[Q_{k,\lambda_{0}}(\psi)<0,\]
 for some $\psi\in E_{k, \lambda_{0}}$ as $k$ small.

\bigskip

We now prove (ii). Since in what follows $\lambda = \lambda_1$ is fixed, we drop the subscripts of its dependence for the sake of clarity.

\medskip To reach a contradiction, we assume that there exists a sequence ${k_{n}}$, which converges to $0$ as $n\rightarrow\infty,$ such that $\tau_{2,n}:=\tau_{2}(k_{n})\leq o_n(1)$, where $\tau_{2}$ is defined as:

\begin{align*}
	\tau_2(k)&=\inf\{Q_{k, \lambda}(\psi):\psi\in E_{k, \lambda},\|\psi\|_{L^{2}(\mathbb{S}^{d}(k)\setminus B_{1})}=1\}.  \end{align*}
	
 By a standard minimization procedure, there exists a sequence of functions $\psi_{n}\in H^{1}_{G}(\mathbb{S}^{d}(k_{n})\setminus B_{1})$ with $\|\psi_{n}\|_{H^{1}_{G}(\mathbb{S}^{d}(k_{n})\setminus B_{1})}=1$ satisfying $\int_{\mathbb{S}^{d}(k_{n})\setminus B_{1}}\psi_{n} z_{k_{n}}=0$ solving the eigenvalue problem:
\begin{equation}\label{eq000}
  \begin{cases}
     -\lambda_{1}\Delta\psi_{n}+\psi_{n}-pu_{k_{n}} ^{p-1}\psi_{n}= \tau_{2,n}\psi_{n}&\mbox{in $\mathbb{S}^{d}(k_{n})\setminus B_{1}$ },\\
      \partial_{\nu}\psi_{n}=\frac{(d-1)k_{n}}{\tan(k_{n})}\psi_{n}+\mu_n &\mbox{on $\partial B_1$},
      \end{cases}
\end{equation}
for some $\mu_n\in \R.$

In what follows we just want to pass to the limit in $\psi_n$, $\tau_{2,n}$ and $z_{k_n}$. This will give a contradiction with Proposition \ref{Pr26}, ii). The limit argument is technically intricate and will be divided in several steps.

\medskip

\textbf{Step 1:} Up to a subsequence, $\psi_{n}$ converges weakly (in a sense to be specified) to some $\psi_{0}$ in $H^{1}_{G}(\mathbb{R}^{d}\setminus B_{1})$.

\medskip

First, let us consider $\psi_{n}$ in coordinates $(r,\theta), r\in (1,\frac{\pi}{k_n}),\theta\in\mathbb{S}^{d-1}.$

For any compact set  $K \subset \R^{d} \setminus B_1$, we have that $\| \psi_n \|_{H^1(K)}$ is bounded. Via a diagonal argument, we can take a subsequence $\psi_n$ such that $\psi_n \rightharpoonup \psi_0$ in $H^1_{loc}(\R^d \setminus B_1)$. By using compactness and taking a convenient subsequence, if necessary, we can
assume also that $\psi_n \to \psi_0$ in $L^2_{loc}(\R^d \setminus B_1)$ and also pointwise. We now claim that $\psi_0 \in H^1(\R^d \setminus B_1)$. We prove this by duality, showing that:

$$ \sup \left \{\int_{\R^d \setminus B_1} \nabla \psi_0 \cdot \nabla \xi + \psi_0 \ \xi:\ \xi \in  H^1(\R^d \setminus B_1) \mbox{ with compact support}, \ \|\xi \|_{H^1} \leq 1 \right  \} \leq 1.$$
Observe that:
$$\frac{S_{k_n}(r)^{d-1}}{r^{d-1}} \xi(r, \theta) \to \xi(r, \theta), \ \ \frac{S_{k_n}(r)^{d-1}}{r^{d-1}} \partial_r \xi (r, \theta) \to \partial_r \xi(r, \theta), $$$$ \frac{S_{k_n}(r)^{d-3}}{r^{d-3}} \nabla_{\theta} \xi(r, \theta) \to \nabla_{\theta} \xi(r, \theta),$$
in $L^2(\R^d \setminus B_1)$. Recall also that $\psi_n \rightharpoonup \psi_0 $ in $H^1_{loc}(\R^d \setminus B_1)$. Then,

\begin{align*} \int_{\R^d \setminus B_1} \nabla \psi_0 \cdot \nabla \xi + \psi_0 \  \xi = \lim_{n \to \infty} \int_{\mathbb{S}_{k_n}^d \setminus B_1}  \nabla \psi_n \cdot \nabla \xi + \psi_n \  \xi \\ \leq \lim_{n \to \infty} \| \xi\|_{H^1(\mathbb{S}^d(k_n) \setminus B_1)}  = \| \xi\|_{H^1(\mathbb{\R}^d \setminus B_1)} \leq 1. \end{align*}

For later use we point out that the restriction of $\psi_0$ to $r \in [1,\frac{\pi}{k_n} ]$ belongs to $L^2(\mathbb{S}^2(k_n) \setminus B_1)$,
since

$$ \int_{1}^{\frac{\pi}{k_n}} \int_{\mathbb{S}^{d-1}} S^{d-1}_{k_n}(r) (\psi_0)^2 \leq
\int_{1}^{+\infty} \int_{\mathbb{S}^{d-1}} r^{d-1} (\psi_0)^2 \leq 1.$$
Moreover, by local weak convergence,

\begin{equation} \label{compact bd} \psi_n|_{\partial B_1} \to \psi_0|_{\partial B_1} \mbox{ strongly in } L^2(\partial B_1). \end{equation}

\medskip

\textbf{Step 2:} The sequences $\tau_{2,n}$, $\mu_n$ are bounded. In particular, $\tau_{2,n}\rightarrow {\tau}_{2} \leq 0$ and $\mu_n \to \mu_0$ up to a subsequence.

\medskip

We argue by a contradiction and assume that $\tau_{2,n}\rightarrow -\infty.$
Multiply the equation (\ref{eq000}) with $\psi_{n},$
\begin{align}\label{eq0003}
 &\int_{\mathbb{S}^{d}(k_{n})\setminus B_{1}} \lambda_{1} |\nabla \psi_{n}|^{2}+ \psi_{n}^{2}-\int_{\mathbb{S}^{d}(k_{n})\setminus B_{1}} pu_{k_{n}}^{p-1}\psi_{n}^{2}-\lambda_{1}\frac{(d-1)k_n}{\tan(k_n)}  \int_{\partial B_{1}}\psi_{n}^2\notag\\
 &=\tau_{2,n}\int_{\mathbb{S}^{d}(k_{n})\setminus B_{1}}\psi_{n}^{2}.
\end{align}
It is clear that the left-hand side of (\ref{eq0003}) is bounded, so $\|\psi_{n}\|$ is supposed to converge to $0$. Therefore, by Lemma \ref{ultimo},
\begin{align*}
\int_{\mathbb{S}^{d}(k_{n})\setminus B_{1}} pu_{k_{n}}^{p-1}\psi_{n}^{2}
\leq C \|\psi_{n}\|_{L^{2}}^{2}
\rightarrow 0,
\end{align*}
Moreover, by the Lemma \ref{Le32},
 \[ \int_{\partial B_{1}}\psi_{n}^2\rightarrow0.\]
Then we can see that left-hand side of \eqref{eq0003} converges to $\lambda_{1}>0$ but the right-hand side is non-positive, which gives a contradiction.

For the estimate on $\mu_n$, take a test function $\phi(r)= (2-r)^+$. Clearly $\phi$ has compact support, is bounded in $H^1(\mathbb{S}^d (k_{n})\setminus B_1)$ and $\phi=1$ on $\partial B_1$. Multiplying equation \eqref{eq000} by $\phi$ and integrating, we obtain

$$\int_{\mathbb{S}^d(k_{n})\setminus B_1} \nabla \psi_n \cdot \nabla \phi + (1-\tau_{2,n}) \psi_n \phi - p u_{k_n}^{p-1} \psi \phi = \mu_n |\partial B_1|.$$
Taking into account the $L^{\infty}$ bound of $u_{k_n}$ (see Lemma \ref{ultimo}), we conclude.

\medskip

\textbf{Step 3:}  $\psi_0$ is a nonzero weak solution of the problem:

\begin{equation}\label{eq0}
	\begin{cases}
		-\lambda_{1}\Delta\psi_{0}+\psi_{}-p\tilde{u} ^{p-1}\psi_{0}= \tau_{2}\psi_{0}&\mbox{in $\mathbb{R}^{d}\setminus B_{1}$ },\\
		\partial_{\nu}\psi_{0}=(d-1) \psi_{0} +\mu_0 &\mbox{on $\partial B_1$}.
	\end{cases}
\end{equation}

Multiplying \eqref{eq000} by a test function with compact support and passing to the limit, we obtain that $\psi_0$ is a solution of \eqref{eq0}. We now show that $\psi_0 \neq 0$.

\medskip

By multiplying the equation (\ref{eq000}) with $\psi_{n}$, we obtain
\begin{align}\label{preq040}
	&\int_{\mathbb{S}^{d}(k_{n})\setminus B_{1}} \lambda_{1} |\nabla \psi_{n}|^{2}+ (1-\tau_{2,n})\psi_{n}^{2}\nonumber\\
	&\quad-\int_{\mathbb{S}^{d}(k_{n})\setminus B_{1}} pu_{k_{n}}^{p-1}\psi_{n}^{2}-\lambda_{1}\frac{(d-1)k_n}{\tan(k_{n})}  \int_{\partial B_{1}}\psi_{n}^2=0,
\end{align}

By \eqref{compact bd}, we have that:

$$ \int_{\partial B_{1}}\psi_{n}^2 \to \int_{\partial B_{1}}\psi_{0}^2.$$

We now claim that

\begin{equation} \label{puff} \displaystyle \int_{\mathbb{S}^{d}(k_{n})\setminus B_{1}}pu_{k_{n}}^{p-1} (\psi_{n}^2 - \psi_0^2) \to 0.\end{equation} Let us fix an arbitrary $\delta>0$, and $M>0$ sufficiently large as in Lemma \ref{ultimo}. We can compute:

\begin{align*}
	&\left | \int_{\mathbb{S}^{d}(k_{n})\setminus B_{1}}  u_{k_{n}}^{p-1} (\psi_{n}- \psi_0)^2  \right | = \left | \int_{1}^{\frac{\pi}{k_{n}}} \int_{\mathbb{S}^{d-1}} S^{d-1}_{k_n}(r)  u_{k_{n}}^{p-1} (\psi_{n}- \psi_0)^2   \right |\\
	& \leq  \left | \int_{1}^{M} \int_{\mathbb{S}^{d-1}} S^{d-1}_{k_n}(r)  u_{k_{n}}^{p-1} (\psi_{n}- \psi_0)^2 \right |  + \left | \int_{M}^{\frac{\pi}{k_{n}}} \int_{\mathbb{S}^{d-1}} S^{d-1}_{k_n}(r)  u_{k_{n}}^{p-1} (\psi_{n}- \psi_0)^2    \right |\\
	&=: (I)+(II).
\end{align*}

We are first concerned with the estimate of (I).
\begin{align*}
	& \left | \int_{1}^{M} \int_{\mathbb{S}^{d-1}} S^{d-1}_{k_n}(r)  u_{k_{n}}^{p-1} (\psi_{n}- \psi_0)^2  \right | \le C \left | \int_{1}^{M} \int_{\mathbb{S}^{d-1}} r^{d-1}  (\psi_{n}- \psi_0)^2 \right | = o_n(1),
\end{align*}
by the $L^2_{loc}$ convergence of $\psi_n$ and the uniform $L^{\infty}$ bound on $u_{k_n}$, see Lemma \ref{ultimo}.

\medskip

We now estimate (II) by making use of Lemma \ref{ultimo}, 2), to obtain:

\begin{align*}
	&  \int_{M}^{\frac{\pi}{k_n}} \int_{\mathbb{S}^{d-1}} S^{d-1}_{k_n}(r)  u_{k_{n}}^{p-1} (\psi_{n}- \psi_0)^2  \le \delta^{p-1}  \int_{M}^{\frac{\pi}{k_n}} \int_{\mathbb{S}^{d-1}} S^{d-1}_{k_n}(r) (\psi_{n}^2+ \psi_0^2) \\ & \le C \delta^{p-1}  \Big ( \int_{M}^{\frac{\pi}{k_n}} \int_{\mathbb{S}^{d-1}} S^{d-1}_{k_n}(r) \psi_{n}^2 + \int_{M}^{+\infty} \int_{\mathbb{S}^{d-1}} r^{d-1} \psi_{0}^2   \Big ) \le 2 C  \delta^{p-1}.
\end{align*}

Since $\delta>0$ is arbitrary, we conclude the proof \eqref{puff}.

\medskip Then, if $\psi_0=0$, \eqref{preq040} implies that:

$$ \int_{\mathbb{S}^{d}(k_{n})\setminus B_{1}} \lambda_{1} |\nabla \psi_{n}|^{2}+ (1-\tau_{2,n})\psi_{n}^{2} \to 0,$$

which is in contradiction with $\| \psi_n \|_{H^1(\mathbb{S}^d(k_n) \setminus B_1) } =1$.

\medskip

\textbf{Step 4:}  Defining $z_n= z_{k_n}$, we show that $\|z_{n}- \tilde{z}\|_{H^1_{0,r}(\mathbb{S}^{d}(k_{n})\setminus B_{1})}\rightarrow 0.$
\medskip

By definition,

\begin{equation}\label{eigen}
	\begin{cases}
		-\lambda_{1}\Delta z_{n}+z_{n}-pu_{k_{n}} ^{p-1} z_{n}= \tau_{n} z_{n}&\mbox{in $\mathbb{S}^{d}(k_{n})\setminus B_{1}$ },\\
		z_{n}=0 &\mbox{on $\partial B_1$},
	\end{cases}
\end{equation}
where $\tau_n <0$ by Proposition \ref{index}. Recall moreover that $\|z_{n}\|_{H^{1}(\mathbb{S}^{d}(k_{n})\setminus B_{1})}=1$. Hence we can use the same ideas in Steps 1, 2 and 3 to prove that $\tau_n \to \tilde{\tau}$ and also $z_{k_{n}}$ converges to $\tilde{z}$ weakly in $H_{0, loc}^1(\R^d \setminus B_1)$, strongly in $L_{loc}^2(\R^d \setminus B_1)$. Moreover $\tilde{z} \neq 0$ belongs to $H_{0}^1(\R^d \setminus B_1)$ with norm smaller or equal than 1, and $\tilde{z}$ is a solution of \eqref{z}.

Observe that the restriction of $\tilde{z}_n$ to $r \in [1, \frac{\pi}{k_n}]$ can be seen as an axially symmetric function in $H^1(\mathbb{S}^{d}(k_n) \setminus B_1)$, since:

$$\| \tilde{z}\|^2_{H^{1}(\mathbb{S}^{d}(k_{n})\setminus B_{1})} = \omega_d \int_1^{\frac{\pi}{k_n}} (\tilde{z}'(r)^2 + \tilde{z}(r)^2) S_{k_n}(r)^{d-1} \leq \omega_d \int_1^{+\infty} (\tilde{z}'(r)^2 + \tilde{z}(r)^2) r^{d-1} \leq 1.  $$

\medskip

{\bf Claim 1:}  $\ \displaystyle\int_{\mathbb{S}^{d}(k_{n})\setminus B_{1}} \lambda_{1} \nabla z_{n} \cdot \nabla (z_n - \tilde{z}) + (1-\tau_{n})z_{n}(z_n - \tilde{z}) = o_n(1).$

\medskip

By multiplying the equation (\ref{eigen}) by $z_{n}-\tilde{z}$ and integrating, we obtain
\begin{align}\label{eq040}
 &\int_{\mathbb{S}^{d}(k_{n})\setminus B_{1}} \lambda_{1} \nabla z_{n} \cdot \nabla (z_n - \tilde{z}) + (1-\tau_{n})z_{n}(z_n - \tilde{z}) -\int_{\mathbb{S}^{d}(k_{n})\setminus B_{1}} pu_{k_{n}}^{p-1} z_{n} (\tilde{z} - z_n)=0.
\end{align}
Hence, it suffices to show that:

$$ \displaystyle \int_{\mathbb{S}^{d}(k_{n})\setminus B_{1}}pu_{k_{n}}^{p-1} z_n (z_{n} - \tilde{z}) \to 0.$$
\medskip

Indeed, by H\"{o}lder inequality and Lemma \ref{ultimo},

\begin{align*} & \displaystyle \int_{\mathbb{S}^{d}(k_{n})\setminus B_{1}}pu_{k_{n}}^{p-1} z_n (z_{n} - \tilde{z})  \le \Big ( \displaystyle \int_{\mathbb{S}^{d}(k_{n})\setminus B_{1}}pu_{k_{n}}^{p-1} (z_{n} - \tilde{z})^2 \Big )^{1/2} \Big ( \displaystyle \int_{\mathbb{S}^{d}(k_{n})\setminus B_{1}}pu_{k_{n}}^{p-1} z_{n}^2 \Big )^{1/2} \\ & \le  \Big ( \displaystyle \int_{\mathbb{S}^{d}(k_{n})\setminus B_{1}}pu_{k_{n}}^{p-1} (z_{n} - \tilde{z})^2 \Big )^{1/2} C \|z_n\|_{L^2}  \end{align*}

Moreover, the same argument of the proof of \eqref{puff} implies that

$$  \displaystyle \int_{\mathbb{S}^{d}(k_{n})\setminus B_{1}}pu_{k_{n}}^{p-1} (z_{n} - \tilde{z})^2 \to 0. $$

\medskip

{\bf Claim 2: } $ \displaystyle \int_{\mathbb{S}^{d}(k_{n})\setminus B_{1}} \nabla \tilde{z} \cdot \nabla (z_{n}-\tilde{z}) \to 0, \ \int_{\mathbb{S}^{d}(k_{n})\setminus B_{1}} \tilde{z}(z_{n} - \tilde{z})\to 0.$

In order to show the first convergence, write:

\begin{align*}  & \int_{\mathbb{S}^{d}(k_{n})\setminus B_{1}} \nabla \tilde{z} \cdot \nabla (z_{n}-\tilde{z})\\ &= \int_{1}^{M} \int_{\mathbb{S}^{d-1}} S^{d-1}_{k_n}(r)  \partial_r \tilde{z} \, \partial_r (z_n - \tilde{z}) + \int_{M}^{\frac{\pi}{k_n}} \int_{\mathbb{S}^{d-1}} S^{d-1}_{k_n}(r)  \partial_r \tilde{z} \, \partial_r (z_n - \tilde{z}) = (I) + (II). \end{align*}

Clearly, $(I) \to 0$ because of the weak convergence $H^1_{loc}$ of $z_n$ to $\tilde{z}$. By H\"{o}lder inequality,

$$ |(II)| \leq C  \Big( \int_{M}^{\frac{\pi}{k_n}} \int_{\mathbb{S}^{d-1}} S^{d-1}_{k_n}(r) |\partial_r \tilde{z}|^2  \Big)^{1/2}.$$
Moreover,

$$ \limsup_{n \to + \infty} \int_{M}^{\frac{\pi}{k_n}} \int_{\mathbb{S}^{d-1}} S^{d-1}_{k_n}(r) |\partial_r \tilde{z}|^2  \leq \int_{M}^{+\infty} \int_{\mathbb{S}^{d-1}} r^{d-1} |\partial_r \tilde{z}|^2,$$
which can be made arbitrarily small by choosing $M$ appropriately.

The second convergence of claim 2 can be justified in the same way.

\bigskip We are now in conditions of proving Step 4:

\begin{align*} & \int_{\mathbb{S}^{d}(k_{n})\setminus B_{1}} \lambda_{1} |\nabla (z_{n} -\tilde{z})| ^2 + (1-\tau_{n}) (z_{n}- \tilde{z})^2 \\= &  \int_{\mathbb{S}^{d}(k_{n})\setminus B_{1}} \lambda_{1} \nabla (z_{n} -\tilde{z}) \cdot \nabla z_n + (1-\tau_{n}) (z_{n}- \tilde{z}) z_n  \\ - &  \int_{\mathbb{S}^{d}(k_{n})\setminus B_{1}} \lambda_{1} \nabla (z_{n} -\tilde{z}) \cdot \nabla \tilde{z} + (1-\tau_{n}) (z_{n}- \tilde{z}) \tilde{z}= o_n(1), \\ \end{align*}

by claims 1 and 2.

\medskip

\textbf{Step 5:} Conclusion.
We now show that
\[\int_{\mathbb{R}^{d}\setminus B_{1}} \psi_{0}\tilde{z}=0.\]
Indeed,
\begin{align*}
0=&\int_{\mathbb{S}^{d}(k_{n})\setminus B_{1}}  \psi_{n}z_{n}\\
&=\int_{\mathbb{S}^{d}(k_{n})\setminus B_{1}}  \psi_{n}(z_{n} -\tilde{z} ) +\int_{\mathbb{S}^{d}(k_{n})\setminus B_{1}}\tilde{z}(\psi_n- \psi_{0}) +\int_{\mathbb{S}^{d}(k_{n})\setminus B_{1}}   \psi_{0}\tilde{z} \\
& = (I) + (II) + (III).
\end{align*}

By H\"{o}lder inequality,

$$ (I) \leq \Big ( \int_{\mathbb{S}^{d}(k_{n})\setminus B_{1}} (z_n - \tilde{z})^2 \Big )^{1/2} \Big ( \int_{\mathbb{S}^{d}(k_{n})\setminus B_{1}} (\psi_n)^2 \Big )^{1/2} \to 0.$$

Moreover, taking $M>0$ sufficiently large,

\begin{align*} (II)= \int_{1}^{M} \int_{\mathbb{S}^{d-1}} S^{d-1}_{k_n}(r) \tilde{z} (\psi_{n}- \psi_0) + \int_{M}^{\frac{\pi}{k_n}} \int_{\mathbb{S}^{d-1}} S^{d-1}_{k_n}(r)  \tilde{z} (\psi_{n}- \psi_0)\end{align*}
The first term above converges to $0$ by the local weak convergence of $\psi_n$ to $\psi_0$. Moreover, the second term can be estimated via H\"{o}lder inequality:

\begin{align*} &\int_{M}^{\frac{\pi}{k_n}} \int_{\mathbb{S}^{d-1}} S^{d-1}_{k_n}(r)  \tilde{z} (\psi_{n}- \psi_0) \\ \le & \Big (\int_{M}^{\frac{\pi}{k_n}} \int_{\mathbb{S}^{d-1}} S^{d-1}_{k_n}(r)  (\psi_{n}- \psi_0)^2 \Big )^{1/2} \Big ( \int_{M}^{\frac{\pi}{k_n}} \int_{\mathbb{S}^{d-1}} S^{d-1}_{k_n}(r)  \tilde{z}^2 \Big )^{1/2}  \\  \le  & C \Big ( \int_{M}^{\frac{\pi}{k_n}} \int_{\mathbb{S}^{d-1}} S^{d-1}_{k_n}(r)  \tilde{z}^2 \Big )^{1/2},\end{align*}
which can be made arbitrarily small if $M$ is sufficiently large.

As a consequence,

$$(III) = \int_{\mathbb{S}^{d}(k_{n})\setminus B_{1}}   \psi_{0}\tilde{z}  \to 0.$$

Passing to the limit, we conclude that

$$ \int_{\mathbb{R}^{d}\setminus B_{1}}   \psi_{0}\tilde{z} =0.$$

This, together with step 3 and the inequality $\tau_2 \leq 0$, gives a contradiction with Proposition \ref{Pr26} (ii).

\bigskip The proof of iii) follows the same arguments as that of ii). In some places the computations are easier since there are no boundary terms. The details are left to the reader.

\end{proof}

Let us now define the quadratic form associated to $H_{k,\lambda}$, namely:
\[ {J_{k,\lambda}: C^{k,\alpha}_{G,0}(\mathbb{S}^{d-1}) \to \R, \ J_{k,\lambda} (v) = \int_{\mathbb{S}^{d-1}}v H_{k, \lambda}(v). }\]
In this case also, sometimes we will drop the subindices $k,\lambda$.
Let us also denote the first eigenvalue of the operator $H_{k, \lambda}$ as
 \[\sigma_{1}(H_{k, \lambda})=\inf\Bigg\{J_{k,\lambda}(v ): v\in C^{k,\alpha}_{G,0}(\mathbb{S}^{d-1})~,~~\int_{\mathbb{S}^{d-1}}v^{2}=1\Bigg\}.\]
By the divergence formula, one can get
\[J_{k,\lambda} (v)=\frac{1}{\lambda}Q_{k, \lambda}(\psi_{v}).\]

\medskip

Next lemma characterizes the eigenvalue $\sigma_{1}(H_{k, \lambda})$ in terms of the quadratic form $Q_{k,\lambda}$.
\begin{lemma} \label{le43}
For any $\lambda\in[\lambda_{0},\lambda_{1}]$, we have
\[\sigma_{1}(H_{k, \lambda})= \min \Bigg\{\frac{1}{\lambda}Q_{k, \lambda}(\psi): \psi\in E_{k, \lambda},~\int_{\partial B_1}\psi^{2}=1\Bigg\}.\]  Moreover the infimum is attained.
 \end{lemma}
\begin{proof}
Let us define
\begin{equation}\label{eq323}
  \eta:=\inf\Bigg\{Q_{k, \lambda}(\psi):\psi\in E_{k,\lambda},\int_{\partial B_1}\psi^{2}= 1 \Bigg\} \in [-\infty, +\infty).
\end{equation}
We first show that $\eta$ is achieved. On that purpose, let us take $\psi_{m}\in E_{k,\lambda}$ such that $Q_{k, \lambda}(\psi_m)\rightarrow\eta .$
We claim that $\psi_{m}$ is bounded. By contradiction, if $\|\psi_{m}\|_{H_{G}^1}\rightarrow +\infty$, we define $\xi_{m}=\|\psi_{m}\|_{H_{G}^1}^{-1}\psi_{m};$ we can suppose that up to a subsequence $\xi_{m}\rightharpoonup \xi_{0}.$ Notice that $\int_{\partial B_1}\xi_{m}^{2}\rightarrow0,$ which yields that $\xi_{0}\in H^{1}_{0,G}(\mathbb{S}^{d}(k)\setminus B_{1}).$ We also point out that
\[\int_{\mathbb{S}^{d}(k)\setminus B_{1}}u_{k,\lambda}^{p-1}\xi_{m}^{2}\rightarrow\int_{\mathbb{S}^{d}(k)\setminus B_{1}}u_{k,\lambda}^{p-1}\xi_{0}^{2}\,.\]
Let us distinguish two cases:\\
\textbf{Case 1:} $\xi_{0}=0.$ In this case
\[Q_{k, \lambda}(\psi_m)=\|\psi_{m}\|_{H_{G}^1}^{2}\int_{\mathbb{S}^{d}(k)\setminus B_{1}}\big(\lambda|\nabla\xi_{m}|^{2}+\xi_{m}^{2}-pu_{k,\lambda}^{p-1}\xi_{m}^{2}\big)-\lambda\frac{(d-1)k}{\tan(k)}\rightarrow+\infty\,,\]
which is impossible.\\
\textbf{Case 2:} $\xi_{0}\neq 0.$ In this case
\begin{align*}
\mathop {\liminf}\limits_{m\rightarrow\infty}Q_{k, \lambda}(\psi_m)&=\mathop {\liminf}\limits_{m\rightarrow\infty}
\|\psi_{m}\|_{H_{G}^1}^{2}\int_{\mathbb{S}^{d}(k)\setminus B_{1}}\big(\lambda|\nabla\xi_{m}|^{2}+\xi_{m}^{2}-pu_{k,\lambda}^{p-1}\xi_{m}^{2}\big)-\lambda\frac{(d-1)k}{\tan(k)}\\
&\geq\mathop {\liminf}\limits_{m\rightarrow\infty}\|\psi_{m}\|_{H_{G}^1}^{2}Q_{k, \lambda}^D(\xi_{0})-\lambda\frac{(d-1)k}{\tan(k)}\,,
\end{align*}
but $Q_{k, \lambda}^D(\xi_{0})>0$ for $\lambda\in[\lambda_{0},\lambda_{1}]$ by the Proposition \ref{Pr25}. This is again a contradiction.

Thus, $\psi_{m}$ is bounded, so up to a subsequence we can pass to the weak limit $\psi_{m}\rightharpoonup\psi.$ Then, $\psi$ is a minimizer for $Q_{k, \lambda}$  and in particular $\eta>-\infty.$

\medskip

By the Lagrange multiplier rule, there exist real numbers $\theta_{0},\theta_{1},\theta_{2}$  so that for any $\rho\in H_{G}^{1}(\mathbb{S}^{d}(k)\setminus B_{1})$,
\[\int_{\mathbb{S}^{d}(k)\setminus B_{1}}\left(\lambda\nabla\psi\nabla\rho+\psi\rho-pu_{k,\lambda}^{p-1}\psi\rho -\theta_{0} \rho z_{k,\lambda}\right)=\int_{\partial B_1}\rho((\theta_{1}+c)\psi+\theta_{2}),\]
where $c=\lambda\frac{(d-1)k}{\tan(k)}.$ Taking $\rho=z_{k,\lambda}$ above we conclude that $\theta_{0}=0$. Moreover, if we take $\rho=\psi$ and $\rho=\vartheta$ (given by (\ref{eq355})), we conclude that $\theta_{1}+c=\eta$ and $\theta_{2}=0$, respectively. In other words, $\psi$
 is a (weak) solution of
\begin{equation}\label{eq325}
  \begin{cases}
  -\lambda\Delta\psi+\psi-pu_{k,\lambda}^{p-1}\psi=0 &\mbox{in $\mathbb{S}^{d}(k)\setminus B_{1}$\,, }\\
  \partial_{\nu}\psi=\eta\psi &\mbox{on $\partial B_1$\,.}
  \end{cases}
\end{equation}
By the regularity theory, $\psi\in C^{2,\alpha}_{G}(\mathbb{S}^{d}(k)\setminus B_{1})$. Define $v = \psi|_{\partial B_1}$ and $\psi\in C^{2,\alpha}_{G}(\mathbb{S}^{d}(k)\setminus B_{1})\cap E_{k,\lambda}$ by the Lemma \ref{Le31}. Observe that:
$$
\int_{\mathbb{S}^{d-1}} v^2 = 1, \ J_{k,\lambda}(v)= \frac{1}{\lambda} Q_{k,\lambda}(\psi)= \frac{1}{\lambda} (\eta-c)\,.
$$
The proof is completed.
\end{proof}

For any $k \in (0, k_0)$, we define
\begin{equation}\label{eq401}
\lambda_{\ast}(k):=\sup\left\{\lambda\in[\lambda_{0},\lambda_{1}]:Q_{k, \lambda}(\psi)< 0 ~\mbox
{for some}~\psi\in E_{k, \lambda}\right\}.
\end{equation}
The above set is non empty since $\lambda_0 $ belongs to it, so the supremum is well defined. It is clear that $\lambda_{0}<\lambda_{\ast}(k)<\lambda_{1}$ by the Proposition \ref{Pr25}. We now can state the main result of this section as the following:
\begin{proposition} \label{Pr41}
We have:
\begin{itemize}
		\item[(i)] if $\lambda=\lambda_{1},$ then $\sigma_{1}(H_{k, \lambda})>0;$
        \item[(ii)] if $\lambda \geq \lambda_{\ast}(k),$ then $\sigma_{1}(H_{k, \lambda}) \geq 0;$
        \item[(iii)] if $\lambda = \lambda_{\ast}(k),$ then $\sigma_{1}(H_{k, \lambda}) = 0;$
		\item[(iv)] for any $\e>0$ there exists $\lambda \in (\lambda_{\ast}- \e, \lambda_{\ast}),$ with $\sigma_{1}(H_{k, \lambda})<0.$
	\end{itemize}
 \end{proposition}
\begin{proof} This proposition follows at once from the Lemma \ref{le43}, Proposition \ref{Pr25} and from the definition of $\lambda_{\ast}$ given by (\ref{eq401}).
 \end{proof}

\section{Bifurcation argument}
\label{Section 6}

In this section, we are ready to prove our main result, Theorem ~\ref{Th11}, by the bifurcation argument. For the sake of completeness, we will now recall the Krasnoselskii bifurcation theorem. For the proof and for many other applications we refer to ~\cite{K04,S94}.

\begin{theorem} \label{Th401} \textbf{\mbox{(Krasnoselskii Bifurcation Theorem)}}
Let $\mathcal{Y}$ be a Banach space, and let $\mathcal{W}\subset \mathcal{Y}$ and $\Gamma\subset\mathbb{R}$ be open sets, where we assume $0\in\mathcal{W}$. Denote the elements of $\mathcal{W}$ by $w$ and the elements of $\Gamma$ by $\lambda$. Let $F:\mathcal{W}\times\Gamma\rightarrow \mathcal{Y}$ be a $C^{1}$ operator such that
\begin{itemize}
  \item[i)] $F(0,\lambda)=0$ for all $\lambda\in\Gamma;$
  \item[ii)] $F(w,\lambda)=w-K(w,\lambda)$, where $K(w,\lambda)$ is a compact map;
  \item[iii)] We denote by $i(\lambda) $ the index of $D_{w}F(0,\lambda)$, i.e., the sum of the multiplicities of all negative eigenvalues of $D_{w}F(0,\lambda)$. Then, there exists $\bar{\lambda} <\hat{\lambda}$ such that:
  \begin{enumerate}
  	\item $D_{w}F(0,\bar{\lambda})$, $D_{w}F(0,\hat{\lambda})$ are non degenerate.
  	\item $i(\bar{\lambda})$ and $i(\hat{\lambda})$ have different parity.
  \end{enumerate}

\end{itemize}
Then there exists $\lambda_{*} \in (\bar{\lambda}, \hat{\lambda})$ a bifurcation point for $F(w,\lambda)=0$ in the following sense: $(0,\lambda_{*})$ is a cluster point of nontrivial solutions $(w,\lambda)\in \mathcal{Y}\times\mathbb{R},w\neq0,$ of $F(w,\lambda)=0.$
  \end{theorem}

\begin{remark} Let us point out that the above version of the Krasnoselskii theorem is not the standard one, as usually one imposes the existence of an isolated point where $D_wF(0, \lambda)$ is degenerate. However, the proof of the theorem works equally well for the version stated above. See \cite[Remark 6.3]{RRS20} on this regard.
\end{remark}

    In order to prove our result, we reformulate the problem. For this aim, we need the following lemma:
     \begin{lemma} \label{le42}
    There exists $\epsilon>0$ such that for any $\lambda\in(\lambda_{*}(k)-\epsilon,\lambda_{1}),$ the operator
     \[H_{k, \lambda}+Id:C^{2,\alpha}_{G,0}(\S^{ d-1})\rightarrow C^{1,\alpha}_{G,0}(\S^{ d-1}),\,v\mapsto H_{k, \lambda}(v)+v, \]
     is invertible.
     \end{lemma}
     \begin{proof}
 { Observe that, by choosing $\epsilon>0$ sufficiently small, we can assume that $\sigma_1(H_{k, \lambda}) > -1$.} 
     We define the quadratic forms $\hat{Q}_{k,\lambda}:E_{k,\lambda}\rightarrow\mathbb{R}$ by
     \[\hat{Q}_{k,\lambda}(\psi)=Q_{k,\lambda}(\psi)+\lambda\int_{\partial B_{1}}\psi^{2},\]
    and $\hat{J}_{k,\lambda}:C^{2,\alpha}_{G,0}(\S^{ d-1})\rightarrow\mathbb{R}$ by
     \[\hat{J}_{k,\lambda}(v)= \int_{\S^{ d-1}} \left( v\partial_{\nu}\psi_{v}-\frac{(d-1)k}{\tan k}v^2+ v^{2}\right).\]
      Actually, these quadratic forms are positive definite since $\sigma_1>-1$. We state that they are also coercive. Let's first show that
     \[\beta:=\inf\{\hat{Q}_{k,\lambda}(\psi):\, \psi\in E_{k,\lambda},\,\|\psi\|=1\}>0\]
     is achieved. On that purpose, take $\psi_{n}\in E_{k,\lambda}$ with $\|\psi_{n}\|_{H_{G}^1}=1$ such that $\hat{Q}_{k,\lambda}(\psi_{n})\rightarrow\beta,$ and suppose that $\psi_{n}\rightharpoonup\psi_{0}.$ If the convergence is strong, then the infimum $\beta$ is attained, which implies that $\beta>0.$ Otherwise,
    \begin{align*}
    \beta &=\mathop{\limsup}\limits_{n\rightarrow\infty}\int_{\S^{d}(k)\setminus B_{1}}\big(\lambda|\nabla\psi_{n}|^{2}
     +\psi_{n}^{2}-pu_{k,\lambda}^{p-1}\psi_{n}^{2}\big)-\lambda\frac{(d-1)k}{\tan k}\int_{\partial B_{1}}\psi_{n}^{2}+ \lambda\int_{\partial B_{1}}\psi_{n}^{2}\\
       &>\int_{\S^{d}(k)\setminus B_{1}}\big(\lambda|\nabla\psi_{0}|^{2}+\psi_{0}^{2}-pu_{k,\lambda}^{p-1}\psi_{0}^{2}\big)-\lambda\frac{(d-1)k}{\tan k}\int_{\partial B_{1}}\psi_{0}^{2}+\lambda \int_{\partial B_{1}}\psi_{0}^{2}\geq0.
\end{align*}
Therefore, $\hat{Q}_{k,\lambda}$ is coercive. Thus, we can obtain that
\begin{align*}\hat{J}_{k,\lambda} (v) =\frac{1}{\lambda}\hat{Q}_{k,\lambda}(\psi_{v})
\geq c\|\psi_{v}\|^{2}_{H^{1}_{G}(\S^{d}(k)\setminus B_{1})}. 
\end{align*}

Observe that $\hat{J}_{k,\lambda}$ is naturally defined in the space: $$X= \left\{u \in H_G^{1/2}(\S^{ d-1}): \int_{\S^{ d-1}} u =0\right\},$$ where $H_G^{1/2}(\S^{ d-1})$ denotes the Sobolev space of $G$-symmetric functions, i.e.
\[
  H^{1/2}_{G}(\S^{ d-1})=\{v\in H^{1/2}(\S^{ d-1}): \ v(\theta)=v(g(\theta)) \ \forall g\in G\}\,.
\]
By the trace estimate $\hat{J}_{k,\lambda}$ is coercive in $X$. According to the Lax-Milgram theorem, the regularity theory and the fact that the mean property is preserved, the operator
 \[v\mapsto \partial_{\nu}\psi_{v}\big|_{\partial B_{1}}-\frac{(d-1)k}{\tan k}v +  v\]
 is invertible from $C^{2,\alpha}_{G,0}(\S ^{d-1})$ to $C^{1,\alpha}_{G,0}(\S ^{d-1})$.
   \end{proof}

     By the Proposition \ref{Pr41}, we can take $ \bar{\lambda}\in(\lambda_{0},\lambda_{*}(k))$ sufficiently close to $\lambda_{*}(k)$ such that $\sigma_{1}(H_{\bar{\lambda}})<0.$ Define the operator $Z:\mathcal{U} \times [\bar{\lambda},\lambda_{1}] \to \mathcal{V}$ by
   \begin{equation*}
    Z(v,\lambda)=F_k(v,\lambda)+v,
   \end{equation*}
  where $\mathcal{U}\subset C^{2,\alpha}_{G,0}(\S ^{d-1}),\mathcal{V}\subset C^{1,\alpha}_{G,0}(\S ^{d-1})$ are open neighborhoods of $0$. By the Lemma \ref{le42}, taking $\bar{\lambda}$ close enough to $\lambda_{*}(k)$ so that we can assume that $D_{v}Z(0,\lambda)$ is an isomorphism for all $\lambda\in[\bar{\lambda},\lambda_{1}].$ Then, we can further restrict $\mathcal{U}$ and $\mathcal{V}$ so that $Z(\cdot,\lambda)$ is invertible for all $\lambda\in[\bar{\lambda},\lambda_{1}]$ according to the Inverse Function theorem.

     Now define the operator $W:\mathcal{V}\times[\bar{\lambda},\lambda_{1}] \rightarrow C^{1,\alpha}_{G,0}(\S ^{d-1})$ by $W(v,\lambda)=v-\hat{v}$ with $Z(\hat{v},\lambda)=v.$ By the compactness of the inclusion of $ C^{2,\alpha}_{G,0}(\S ^{d-1})$ into $C^{1,\alpha}_{G,0}(\S ^{d-1}),$ we can point out that $W$ is the operator by the sum of an identity and a compact operator. Obviously, $F_k(v,\lambda)=0\Leftrightarrow W(v,\lambda)=0.$ Theorem \ref{Th11} follows if we show the local bifurcation of solutions to the equation $W(v,\lambda)=0.$

     We have
     \[D_{v}W\big|_{(0,\lambda)}(v)=v-D_{v}Z^{-1}\big|_{(0,\lambda)}(v).\]
     Thus
     \[D_{v}W\big|_{(0,\lambda)}(v)=\delta v\Leftrightarrow H_{k, \lambda}(v)=\frac{\delta}{(1-\delta)\partial_r u_{k,\lambda}(1)}v.\]
    Recall from the proof of the Lemma \ref{le42}, we have that $\delta<1$ if $\lambda\geq\bar{\lambda}$. Hence, $D_{v}W(0,\lambda)$ has the same number of negative eigenvalues as $H_{k, \lambda}.$

Under this framework, theorem ~\ref{Th11} follows immediately from the following lemma and the Krasnoselskii theorem.

\begin{lemma} \label{Le51}
The index of the linearized operator $D_{v}W(0,\lambda)$ is odd for some $\lambda < \lambda_{\ast}(k)$ sufficiently close to $\lambda_{\ast}(k).$
 \end{lemma}
\begin{proof}
In view of Proposition \ref{Pr41}, it is sufficient to prove that $H_{k,\lambda_{\ast}(k)}$ has odd-dimensional kernel. For any $\psi\in E_{k, \lambda},$ there exist functions $\psi_{0},\psi_{l,j}$ defined in $[1,\frac{\pi}{k})$ such that
\[\psi(r,\theta)=\psi_{0}(r)+\sum_{l=1}^{+\infty}\sum_{j=1}^{m_{l}}\psi_{l,j}(r)\xi_{l,j}(\theta),\]
where $(r,\theta)\in [1,\frac{\pi}{k})\times \mathbb{S}^{d-1},$ and $\xi_{l,j}$ are the $G$-symmetric spherical harmonics, normalized to $1$ in the $L^{2}$-norm, with the increasing sequence of eigenvalues $\mu_{i_{l}}$ of multiplicity $m_{l}.$ Then the quadratic form $\psi\mapsto Q_{k, \lambda}(\psi)$ defined in $E_{k, \lambda}$ can be given by
\begin{equation}\label{eq501}
Q_{k, \lambda}(\psi)=Q_{k,\lambda}^{0}(\psi_{0})+\sum_{l=1}^{+\infty}\sum_{j=1}^{m_{l}}Q_{k,\lambda}^{l}(\psi_{l,j}),
\end{equation}
where the functional $Q_{k,\lambda}^{l}$ is defined as
\begin{align*}
Q_{k,\lambda}^{l}(\phi)&=\int_{1}^{\frac{\pi}{k}}(\lambda\phi'^{2}+\phi^{2}-pu_{k,\lambda}^{p-1}\phi^{2})S_{k}^{d-1}dr\\
&\quad-\lambda\frac{(d-1)k}{\tan(k)}\phi(1)^{2}
+\lambda\mu_{i_{l}}\int_{1}^{\frac{\pi}{k}}\phi^{2}S_{k}^{d-3}dr
\end{align*}
for a function $\phi:\left(1,\frac{\pi}{k}\right)\rightarrow\mathbb{R}.$ By convention, we choose $\mu_{i_{0}}=0.$ Also $\psi_{0}(1)=0$ and $\psi_{0}$ is orthogonal to the function $z_{k,\lambda}$ restricted to the radial variable since $\psi\in E_{k,\lambda}.$ We know that $Q_{k,\lambda}^{0}(\psi_{0})>0$ in the radial case. For $\lambda=\lambda_{\ast}(k),$ $Q_{k, \lambda}(\psi)\geq0,$ and then $Q_{k,\lambda}^{l}\geq0$ by (\ref{eq501}). Moreover, it is obvious that
\[Q_{k,\lambda}^{l_1}<Q_{k,\lambda}^{l_2} \,~\mbox {if} \,\,1\leq l_{1}<l_{2}.\]
We also know that there exists a $\psi\in E_{k,\lambda}$ such that $Q_{k, \lambda}(\psi)=0.$ Therefore $Q_{k,\lambda}^{1}\geq0$ and $Q_{k,\lambda}^{l}>0$ for $l>1.$ This implies that the dimension of the kernel of the operator $H_{\lambda_{\ast}(k)}$ is $m_{1},$ which is odd by the assumption (G) on the symmetry group.
 \end{proof}

\section{Appendix}

In this appendix we prove the quantitative version of the Implicit Function Theorem given in Proposition \ref{Pr22}, the uniformity of the Sobolev constant given in Lemma \ref{Le202}, and also the proof of Lemma \ref{Le32}.

\begin{proof}[Proof of Proposition \ref{Pr22}]
	Let us define the map $T:U\rightarrow Y$ by setting
	\[T(w)=w-[F'(v)]^{-1}(F(v+w)).\]
	Clearly, a fixed point $z$ of $T$ will give rise to a solution to the equation $F(v+z)=0.$ We apply now the Banach contraction Theorem to the operator $T.$
	
	For any $\psi\in Y,w\in U,$ one has
	\[T'(w)[\psi]=\psi-[F'(v)]^{-1}(F'(v+w)[\psi])=[F'(v)]^{-1}(F'(v)[\psi]-F'(v+w)[\psi]).\]
	Thus we find
	\begin{equation}\label{eq15}
		\|T'(w)[\psi]\|\leq\frac{c}{2c }\|\psi\| = \frac{1}{2}\|\psi\|.
	\end{equation}
	Therefore we conclude that $T$ is a contraction. We finish the proof if we show that $T$ maps $U$ into itself. With this purpose, let us compute:
	\[\|T(0)\|=\|[F'(v)]^{-1}(F(v))\|\leq c\delta.\]
	On the other hand, for any $w\in U$ we can use (\ref{eq15}) to deduce
	\[\|T(w)-T(0)\|\leq\frac{c}{2c}\|w\|\leq c\delta.\]
	By using the triangular inequality of the norm, we get
	\[\|T(w)\|\leq2c\delta.\]
	By the Banach contraction Theorem, $T$ has a fixed point in $U$. Moreover, by (A2) and (A3) we conclude that $\| F'(v+z)^{-1} \| \leq 2c$.
\end{proof}

\bigskip

\begin{proof}[Proof of Lemma \ref{Le202}]
	We first give the facts that
	\[\|u\|_{L^{s}}\leq C(q)\left(\|\nabla u\|_{L^{q}}+\|u\|_{L^{q}}\right) \,\ \forall\,\ u\in H^{1,q}(\mathbb{S}_{1}^{2}),~\frac{1}{s}=\frac{2-q}{2q},~1<q<2,\]
	and
	\[\|u\|_{L^{2^{\ast}}}\leq C(d)\left(\|\nabla u\|_{L^{2}}+\|u\|_{L^{2}}\right) \,\ \forall\,\ u\in H^{1}(\mathbb{S}_{1}^{d}),d\geq3, \]
	see \cite{A98,B93}. And we expand that
	\begin{equation*}
		\bar{u}=
		\begin{cases}
			u, &x \in \mathbb{S}^{d}(k)\setminus B_{1},\\
			0, &x \in B_{1},
		\end{cases}
	\end{equation*}
	so that $\bar{u}\in H^{1}(\mathbb{S}^{d}(k))=H^{1}(\mathbb{S}_{R}^{d}),R=\frac{1}{k}.$ Taking $\hat{u}(y)=\bar{u}(Ry),y\in \mathbb{S}_{1}^{d}$, we now consider two cases:\\
	\textbf{Case 1:} $d=2.$\\
	Let $f(t)=\frac{2t}{2-t},$ with $f(1)=2,f(2_{-})=+\infty,$ then for any $s\in(2,+\infty),$ there exists a unique $t\in(1,2)$ such that $f(t)=\frac{2t}{2-t}=s.$
	Then we have
	\[\|u\|_{L^{s}}\leq C(\|u\|_{L^{t}}+\|\nabla u\|_{L^{t}})\]
	by the fact that
	\begin{align*}
		\left(\int_{\mathbb{S}_{R}^{2}}|\bar{u}|^{s}\right)^{\frac{t}{s}}&= R^{2-t}\left(\int_{\mathbb{S}_{1}^{2}}|\hat{u}|^{s}\right)^{\frac{t}{s}} \\
		&\leq R^{2-t}C\left(\int_{\mathbb{S}_{1}^{2}}|\hat{u}|^{t}+\int_{\mathbb{S}_{1}^{2}}|\nabla \hat{u}|^{t}\right)\\
		&= R^{2-t}C\left(R^{-2}\int_{\mathbb{S}_{R}^{2}}|\bar{u}|^{t}+R^{t-2}\int_{\mathbb{S}_{R}^{2}}|\nabla \bar{u}|^{t}\right)\\
		&\leq C\left(\int_{\mathbb{S}_{R}^{2}}|\bar{u}|^{t}+\int_{\mathbb{S}_{R}^{2}}|\nabla \bar{u}|^{t}\right),
	\end{align*}
	where $C=C(s)$ is a positive constant. We take $u^{2}$ instead of $u,$ then we can get that
	\begin{align*}
		\|u^{2}\|_{L^{s}}&\leq C(\|u^{2}\|_{L^{t}}+\|\nabla u^{2}\|_{L^{t}})\\
		&\leq C(\|u\|_{L^{2}}+2\|\nabla u\|_{L^{2}})\|u\|_{L^{s}}\\
		&\leq C(\|u\|_{L^{2}}+2\|\nabla u\|_{L^{2}})\|u\|^{1-\beta}_{L^{2}}\|u\|^{\beta}_{L^{2s}},
	\end{align*}
	where $\beta=\frac{2}{s}.$ Then one can get that
	\[\|u\|^{2-\beta}_{L^{2s}}\leq C(\|u\|_{L^{2}}+2\|\nabla u\|_{L^{2}})\|u\|^{1-\beta}_{L^{2}}.\]
	Therefore, for any $s>4,$ we have \[\|u\|_{L^{s}}\leq C\|u\|_{H^{1}}.\]
	As $2<s\leq4,$ we obtain
	\[\|u\|_{L^{s}}\leq\|u\|_{L^{2}}^{\gamma}\|u\|_{L^{2s}}^{1-\gamma},\]
	where $\gamma=\frac{2}{s}.$ Also, we can gain
	\[\|u\|_{L^{s}}\leq C\|u\|_{H^{1}}.\]
	\textbf{Case 2:} $d\geq3.$\\
	By the interpolation inequality, one has that
	\[\|u\|_{L^{s}}\leq\|u\|_{L^{2}}^{\alpha}\|u\|_{L^{2^{\ast}}}^{1-\alpha},\]
	where $\alpha$ satisfies $\frac{\alpha}{2}+(1-\alpha)(\frac{1}{2}-\frac{1}{d})=\frac{1}{s}.$ We also have
	\begin{align*}
		\left(\int_{\mathbb{S}_{R}^{d}}|\bar{u}|^{2^{\ast}}\right)^{\frac{2}{2^{\ast}}}&= R^{d-2}\left(\int_{\mathbb{S}_{1}^{d}}|\hat{u}|^{2^{\ast}}\right)^{\frac{2}{2^{\ast}}} \\
		&\leq R^{d-2}C\left(\int_{\mathbb{S}_{1}^{d}}|\hat{u}|^{2}+\int_{\mathbb{S}_{1}^{d}}|\nabla \hat{u}|^{2}\right)\\
		&= R^{d-2}C\left(R^{-d}\int_{\mathbb{S}_{R}^{d}}|\bar{u}|^{2}+R^{2-d}\int_{\mathbb{S}_{R}^{d}}|\nabla \bar{u}|^{2}\right)\\
		&\leq C\left(\int_{\mathbb{S}_{R}^{d}}|\bar{u}|^{2}+\int_{\mathbb{S}_{R}^{d}}|\nabla \bar{u}|^{2}\right),
	\end{align*}
	where the positive constant $C$ just depends on $d$. Then one can obtain that $\|u\|_{L^{s}}\leq C\|u\|_{H_{0}^{1}}.$ The proof of the lemma is done.
\end{proof}
\bigskip

\begin{proof}[Proof of Lemma \ref{Le32}]

Take a smooth vector field $M(x)$ in $\mathbb{S}^{d}(k)\setminus B_{1}$ such that $M(x)=\nu(x)$ on the boundary $\partial B_{1},|M(x)|\leq1.$ Then, we apply the divergence theorem and H\"{o}lder inequality to get
 \begin{align*}
 & \int_{\partial B_{1}}\psi^2= \int_{\partial B_{1}}\psi^2\cdot M(x)\cdot\nu(x)\\
 &=\int_{\mathbb{S}^{d}(k)\setminus B_{1}} \left(2\psi\nabla \psi\cdot M(x)+\psi^2\cdot div M(x)\right)\\
 &\leq2\|\nabla \psi\|_{L^{2}}\|\psi\|_{L^{2}}+\|\psi\|^{2}_{L^{2}}\| div M(x)\|_{L^{\infty}}\\
 &\leq 2\|\nabla \psi\|_{L^{2}}\|\psi\|_{L^{2}}+C\|\psi\|^{2}_{L^{2}}
\end{align*}
since $| div M(x)|$ is uniformly bounded. In order to prove that, let us now consider the vector field $M(x)$ in the coordinates $(r,\theta),r\in(1,\frac{\pi}{k}),\theta\in\mathbb{S}^{d-1},$ and give by
$$M(r,\theta)=-\frac{\partial}{\partial r}X(r,\theta)\cdot\chi(r)=-\textbf{v}_{r}\cdot\chi(r)=-\left(M_{r}\textbf{v}_{r}+\sum\limits_{i=1}^{d-1}M_{\theta_i}\textbf{v}_{\theta_i}\right),$$
where $\textbf{v}_{r}:=\textbf{v}_{r}(\theta),\textbf{v}_{\theta_i}$ are orthonormal vectors along $X$ and $\chi(r)$ is a cut-off function
\begin{equation*}
	 \chi(r)=
	\begin{cases}
	0, &r \geq\frac{3}{2},\\
	1, &r \leq\frac{5}{4}.
	\end{cases}
	\end{equation*}	
Then the divergence can be written as
 \begin{align*}
 div M(x) &=\frac{1}{S_{k}^{d-1}(r)}\frac{\partial}{\partial_{r}}(S_{k}^{d-1}(r)M_{r}(r,\theta))+\frac{1}{S_{k}(r)}div_{\theta} M_{\theta}(r,\theta)\\
 &=\partial_{r}M_{r}(r,\theta)+(d-1)\frac{C_{k}(r)}{S_{k}(r)}M_{r}(r,\theta)+\frac{1}{S_{k}(r)}div_{\theta} M_{\theta}(r,\theta)\\
 &=-\chi'(r)-(d-1)\frac{C_{k}(r)}{S_{k}(r)}\chi(r).
\end{align*}
Therefore,
 $$|div M(x)|\leq C(d)$$
 as $k\rightarrow0,$ and the proof is finished.
\end{proof}


\end{document}